\documentclass[preprint,3p]{elsarticle}

\usepackage{amssymb,amsmath,amsthm}
\usepackage{mathtools}

\usepackage{array}
\usepackage{tikz-cd}
\usepackage[colorlinks=true]{hyperref}
\usepackage{enumitem}
\setlist[enumerate]{font=\upshape}
 
\theoremstyle{plain}
\newtheorem{thm}{Theorem}[section]
\newtheorem{lem}[thm]{Lemma}
\newtheorem{prop}[thm]{Proposition}
\newtheorem{cor}[thm]{Corollary}
\theoremstyle{definition}
\newtheorem{defn}[thm]{Definition}

\newtheorem{exmp}[thm]{Example}
\newtheorem{rem}[thm]{Remark}
\numberwithin{equation}{section}

\begin{document}

\begin{frontmatter}

\title{d-spectral bitopological spaces}
 
\author[a]{Hang Yang}
\ead{yanghangscu@qq.com}

\author[b]{Dexue Zhang}
\ead{dxzhang@scu.edu.cn}

\affiliation[a]{%
  organization={School of  Mathematics and Physics, Jiangxi University of Science and Technology},
  city={Ganzhou},
  postcode={341000},
  country={China}
}

\affiliation[b]{%
  organization={School of Mathematics, Sichuan University},
  city={Chengdu},
  postcode={610064},
  country={China}
}
 
\begin{abstract}
We introduce the category of d-spectral spaces, a bitopological analogue of the spectral spaces of Stone and Hochster. A d-spectral space is a bitopological space satisfying three conditions: the join of its two topologies is compact, it is d-sober in the sense of Jung and Moshier, and each of its two topologies is a coherent frame. We establish a dual equivalence between this category and the category of d-lattices satisfying the finitary cut rules. The constructions implementing the duality are the d-lattice of compact open sets of a d-spectral space and the d-spectrum construction for d-lattices. We further show that spectral spaces embed as a full subcategory of d-spectral spaces that is both reflective and coreflective. Finally, the patch space of a d-spectral space is d-Boolean, and its de Groot dual is again d-spectral, mirroring the classical spectral case.  
\end{abstract}
 

\end{frontmatter}

\section{Introduction} 

Spectral spaces originated in Marshall Stone's study of topological representations of distributive lattices. In \cite{Stone36,Stone37a}, Stone established a dual equivalence between the category of Boolean algebras and the category of Boolean spaces, that is, compact, Hausdorff, and totally disconnected spaces. He then extended this duality to the category of distributive lattices by topologising the set of prime ideals (equivalently, prime filters) of a distributive lattice; the resulting topological space is now called the \emph{spectrum} of the lattice. Stone isolated certain topological properties of these spectra and proved that every topological space satisfying those properties arises as the spectrum of a distributive lattice, thereby establishing a duality for distributive lattices \cite{Stone37b}.

The spectrum of a distributive lattice is precisely a compact, sober space whose open set lattice is a coherent frame. Stone did not name these spaces; in 1969, Hochster \cite{Hochster} coined the term \emph{spectral spaces} for them and characterized them topologically as the compact, sober spaces whose compact open sets are closed under finite intersections and form a basis. Since their introduction, spectral spaces have grown considerably in significance, finding applications in algebraic geometry, where they arise as the prime spectra of commutative rings, as well as in categorical logic, domain theory, and theoretical computer science. For a comprehensive treatment, we refer the reader to the monograph \cite{DST}.

Two constructions on spectral spaces are of particular importance. The first is the \emph{de Groot dual}, whose topology is generated by taking the compact saturated sets as a subbasis of closed sets; the second is the \emph{patch topology}, obtained by joining the original topology with the de Groot dual. Both constructions reveal deep connections to order theory and to compact Hausdorff spaces: the de Groot dual of a spectral space is again spectral, and the patch space of a spectral space is a Boolean space \cite{Hochster}. Moreover, these constructions are the topological counterparts of algebraic operations, namely taking the opposite lattice and passing to the Boolean reflection.

The aim of this paper is to introduce and develop a bitopological analogue of spectral spaces, which we call \emph{d-spectral spaces}. The key ingredient is the notion of d-sobriety due to Jung and Moshier \cite{JM2006,JM2008}, which arises from their theory of d-frames: structures that model pairs of topologies together with consistency and totality predicates. Precisely, a d-spectral space is a compact and d-sober bitopological space $(X,\tau_+,\tau_-)$ such that both open set lattices $\tau_+$ and $\tau_-$ are coherent frames.

Our first main result shows that this notion fits naturally into the existing theory. Recall that a bitopological space is \emph{d-Boolean} if it is compact and totally order-separated. It is shown in \cite{BBGK} that $\mathbf{dBoolSp}$, the category of d-Boolean spaces, is equivalent to $\mathbf{Spectral}$, the category of spectral spaces and spectral maps; the equivalence sends a spectral space $(X,\tau)$ to $(X,\tau,\tau^k)$. We prove
that d-Boolean spaces are precisely the order-separated d-spectral spaces (Proposition \ref{d-Boolean via d-spectral}) and that, under this equivalence, $\mathbf{Spectral}$ embeds as a full subcategory
of $\mathbf{dSpectral}$ that is both reflective and coreflective (Corollary \ref{d-Boolean r.c.in d-spectral}). Combined with the classical
fact that $\mathbf{BoolSp}$ is both reflective and coreflective in $\mathbf{Spectral}$, this yields the following chain:
\[
\mathbf{BoolSp}\subseteq \mathbf{Spectral}\cong \mathbf{dBoolSp}
\subseteq \mathbf{dSpectral}.
\]
Each of the two inclusions has both a left and a right adjoint; hence both $\mathbf{Spectral}$ and $\mathbf{dSpectral}$ are well-behaved extensions of $\mathbf{BoolSp}$.

Our second main result is an algebraic duality for d-spectral spaces. We show that the d-spectrum construction  introduced in \cite{YZ2026} assigns to every d-lattice a d-spectral space (Theorem \ref{functor dSpec}), and that conversely every d-spectral space arises in this way, as the d-spectrum of its d-lattice of compact open sets (Proposition \ref{spectrum is spectral}). The duality, however, is not with the whole category ${\bf dLat}$ of d-lattices: we identify a condition on d-lattices, the \emph{finitary cut rules} from \cite{JM2006,Klinke}, which is exactly what is needed for the counit of the adjunction $\mathrm{K}\mathcal{O}\dashv\operatorname{dSpec}$ to be an isomorphism. We prove that a d-lattice satisfies the finitary cut rules if and only if its d-frame of ideals is spatial (Lemma \ref{existence of prime d-filter}, Proposition \ref{cut rules implies isomorphism}). This yields a dual equivalence between $\mathbf{dSpectral}$, the category of d-spectral spaces, and $\mathbf{dLat}_{\mathrm{cut}}$, the category of d-lattices satisfying the finitary cut rules (Theorem \ref{the duality}). It is the bitopological counterpart of Stone duality for distributive lattices.

We also show that the standard constructions on spectral spaces carry over to the bitopological setting. The order-dual of a d-lattice corresponds to the de Groot dual of its d-spectrum (Lemma \ref{dspec preserves involution});
hence the de Groot dual of a d-spectral space is again d-spectral. The patch space of a d-spectral space is d-Boolean and coincides with the d-spectrum of the reflection of its d-lattice of compact opens into d-Boolean algebras (Proposition~\ref{reflective}, Theorem \ref{spectrum of Gamma = patch}). 

The paper is organized as follows. Section \ref{review} provides a brief review of spectral spaces for motivation and ease of reference. Section \ref{dS} introduces d-spectral spaces and establishes their basic properties, in particular the characterization of d-Boolean spaces as the order-separated d-spectral spaces. Section \ref{dLB} reviews the necessary material on d-lattices, d-Boolean algebras, and d-frames. 
Section \ref{Spectra} studies the d-spectrum construction and proves that d-spectral spaces are precisely the d-spectra of d-lattices. Section \ref{duality} establishes the duality between the category of d-spectral spaces and d-spectral maps and the category of d-lattices satisfying the finitary cut rules. It also examines properties of d-lattices satisfying the finitary cut rules, and behavior of d-spectral spaces under the de Groot dual and the patch constructions. 

\section{Review of spectral spaces}\label{review}

This section provides a brief review of spectral spaces, serving both as motivation and as a convenient reference. Our standard references are \cite{AHS} for category theory, \cite{DST,Hochster} for the theory of spectral spaces, and \cite{Engelking} for general topology.

Throughout this paper, all distributive lattices are assumed to be bounded, i.e., to have a top element $1$ and a bottom element $0$. To keep the notation consistent with that used in Section \ref{dLB} for the \emph{information order} on d-lattices, we follow Jung and Moshier \cite{JM2006, JM2008} in writing $\sqsubseteq$ for the order of a distributive lattice and $\sqcap, \sqcup$ for its meet and join, while reserving the usual $\leq$, $\wedge$, and $\vee$ for the  order relation, the meet, and the join, respectively, of the \emph{logic order} on the d-lattice. 

An element $a$ of a distributive lattice $(L,\sqsubseteq)$ is \emph{complemented} if there exists $b\in L$ such that
\[
a\sqcup b = 1 ~~\text{and}~~ a\sqcap b = 0.
\]
Such a $b$, when it exists, is necessarily unique and is called the \emph{complement} of $a$. A Boolean algebra is a distributive lattice in which every element is complemented. 

An element $x$ of a partially ordered set $P$ is called \emph{finite} (or \emph{compact}) \cite{Gierz2003,Johnstone} if, for every directed subset $D\subseteq P$,
\[
x \sqsubseteq \bigsqcup D \implies x \sqsubseteq d \text{ for some } d\in D.
\]

A frame $L$ is \emph{coherent} \cite[page 63]{Johnstone} if:
\begin{enumerate}[label=(\roman*)] 
\item every element of $L$ is a join of finite elements; and
\item the finite elements form a sublattice of $L$, i.e., $1$ is finite and the meet of two finite elements is again finite.
\end{enumerate}
Coherent frames are precisely the frames of ideals of distributive lattices; see \cite[page 64]{Johnstone}.

\begin{defn}(Hochster \cite{Hochster})
A topological space $(X,\tau)$ is called \emph{spectral} if it is compact, sober, and its compact open sets are closed under finite intersections and form a basis for $\tau$.
\end{defn}
Equivalently, a topological space $(X,\tau)$ is spectral if it is compact, sober, and its open set lattice $\tau$ is a coherent frame. Although coherence of $\tau$ already implies compactness, we retain compactness in the definition to align it with the corresponding definition of d-spectral spaces in the next section.

A continuous map $f\colon (X,\tau) \to (X',\tau')$ between topological spaces is said to be \emph{spectral} if the preimage of every compact open subset of $(X',\tau')$ is a compact open subset of $(X,\tau)$. The category of spectral spaces and spectral maps is denoted by
\[
{\bf Spectral}.
\]

For a spectral space $(X,\tau)$, the coherence of $\tau$ implies that the collection of compact open subsets of $(X,\tau)$ forms a distributive lattice. Associating to each spectral space this distributive lattice yields a contravariant functor
\[
{K}\colon {\bf Spectral} \to  {\bf DisLat}^{\rm op},
\]
from the category of spectral spaces and spectral maps to the opposite of the category ${\bf DisLat}$ of distributive lattices and lattice homomorphisms. To see that ${K}$ is an equivalence of categories, we require a few additional notions.

A \emph{filter} of a distributive lattice $(L,\sqsubseteq)$ is a non-empty subset $F\subseteq L$ such that:
\begin{itemize} 
\item if $a\in F$ and $a\sqsubseteq b$, then $b\in F$;
\item if $a,b\in F$, then $a\sqcap b\in F$.
\end{itemize}
A filter $F$ is \emph{proper} if $0\notin F$, and \emph{prime} if it is proper and
\[
a\sqcup b\in F \implies a\in F \text{ or } b\in F.
\]
The notions of \emph{ideal}, \emph{proper ideal}, and \emph{prime ideal} are defined dually. 

In any distributive lattice, the complement of a prime ideal is a prime filter, and vice versa.
For a distributive lattice $L$, we denote by $\operatorname{Fil} L$ the set of all filters of $L$, and by $\operatorname{Idl} L$ the set of all ideals of $L$.

Let $\operatorname{Spec} L$ denote the set of all prime filters of $L$. For each $a\in L$, define
\[
\Phi(a)=\{F\in \operatorname{Spec} L : a\in F\}.
\]
The sets $\Phi(a)$ form a basis for a topology on $\operatorname{Spec} L$. The resulting topological space is called the \emph{spectrum} of $L$. Every open subset of $\operatorname{Spec} L$ is of the form
\[
\Phi(I)=\{F\in \operatorname{Spec} L : I\cap F \neq \emptyset\}
\]
for some ideal $I$ of $L$. In particular, the compact open subsets of the spectrum are precisely the sets $\Phi(a)$ for $a\in L$.

The spectrum of any distributive lattice is a spectral space, so we obtain a contravariant functor
\[
\operatorname{Spec}\colon {\bf DisLat}^{\rm op} \to  {\bf Spectral}.
\]
Together, the functors
\[
{K}\colon {\bf Spectral} \to {\bf DisLat}^{\rm op}
\quad\text{and}\quad
\operatorname{Spec}\colon {\bf DisLat}^{\rm op} \to {\bf Spectral}
\]
witness the dual equivalence between the category of distributive lattices and the category of spectral spaces. For details, see \cite{Johnstone} and \cite{GG2024}.

Now let $(X,\tau)$ be a topological space. The \emph{specialization order} of $(X,\tau)$ \cite{Gierz2003} refers to the preorder on $X$ defined by $ x\leq_\tau y\iff x\in\overline{\{y\}}$. A subset of $(X,\tau)$ is called \emph{saturated} \cite{Gierz2003} if it is an upper set with respect to the specialization order. 
The \emph{co-compact topology} $\tau^k$ on $X$ is the topology having the compact saturated subsets of $(X,\tau)$ as a subbasis of closed sets. For more information on the co-compact topology, see \cite{DST,Gierz2003}. The space $(X,\tau^k)$ is called the \emph{de Groot dual} of $(X,\tau)$. The \emph{patch topology} is defined as $\tau\vee\tau^k$, where the join of two topologies on the same set means the smallest topology containing both. The space $(X,\tau\vee\tau^k)$ is called the \emph{patch space} of $(X,\tau)$.

The patch topology and the de Groot dual are natural constructions on spectral spaces that reveal deep connections to orders and compact Hausdorff spaces. 
The following lemma is well known; see, for instance, Section 3.5 of \cite{DST}. The proof included here is intended to motivate the proof of Lemma \ref{dspec preserves involution}.

\begin{lem}\label{spectrum preserves involution}
For every distributive lattice $L$, the spectrum of $L^{\rm op}$ is homeomorphic to the de Groot dual of the spectrum of $L$.
\end{lem}
\begin{proof}
Let $(\operatorname{Spec} L,\tau)$ be the spectrum of $L$. Since the compact open subsets of $(\operatorname{Spec} L,\tau)$ are precisely the sets $\Phi(a)$ for $a\in L$, and these form a basis for $\tau$, it follows easily that every compact saturated subset of $(\operatorname{Spec} L,\tau)$ can be expressed as an intersection of sets of the form $\Phi(a)$. Consequently, the complements $\operatorname{Spec} L\setminus \Phi(a)$ constitute a basis for the co-compact topology $\tau^k$.

We now prove the conclusion. It is easily verified that $F\subseteq L$ is a prime filter of $L$ if and only if $L\setminus F$ is a prime filter of $L^{\rm op}$. Hence the map
\[
F \mapsto L\setminus F
\]
is a bijection from the set of prime filters of $L$ to that of $L^{\rm op}$. Consequently, the spectrum of $L^{\rm op}$ can be identified with the set $\operatorname{Spec} L$, equipped with the topology $\tau^{\rm op}$ generated by the sets
\[
\Phi^{\rm op}(a)=\{F\in \operatorname{Spec} L : a\in L\setminus F\}=\{F\in \operatorname{Spec} L : a\notin F\}, \qquad a\in L.
\]
For all $a\in L$ and $F\in \operatorname{Spec} L$,
\[
F\in \operatorname{Spec} L\setminus \Phi(a) \iff a\notin F \iff F\in \Phi^{\rm op}(a),
\]
so $\tau^{\rm op}=\tau^k$. Hence the spectrum of $L^{\rm op}$ is homeomorphic to the de Groot dual of the spectrum of $L$.
\end{proof}

\begin{thm}\label{patch space of spectral space}
For every spectral space $(X,\tau)$:
\begin{enumerate}[label=(\roman*)] 
\item {\rm (\cite[Proposition 8]{Hochster})} The de Groot dual $(X,\tau^k)$ is a spectral space;
\item {\rm(\cite[Proposition 4]{Hochster})} The patch space $(X,\tau\vee\tau^k)$ is a Boolean space.
\end{enumerate}
\end{thm}

\section{d-spectral spaces} \label{dS}

A \emph{bitopological space} \cite{Kelly} is a triple $(X,\tau_+,\tau_-)$, where $X$ is a set and $\tau_+,\tau_-$ are topologies on $X$. A \emph{continuous} (also called bicontinuous in the literature) map
\[
f\colon (X,\tau_+,\tau_-) \to (X',\tau_+',\tau_-')
\]
is a function $f\colon X\to X'$ such that both
\[
f\colon (X,\tau_+)\to (X',\tau_+')
\quad\text{and}\quad
f\colon (X,\tau_-)\to (X',\tau_-')
\]
are continuous. The category of bitopological spaces and continuous maps is denoted by
\[
{\bf BiTop}.
\]

\begin{defn}\label{compact bitop}
A bitopological space $(X,\tau_+,\tau_-)$ is said to be:
\begin{enumerate}[label=(\roman*)] 
\item $T_0$ if the topological space $(X,\tau_{+}\vee\tau_{-})$ is $T_0$;
\item compact if the topological space $(X,\tau_{+}\vee\tau_{-})$ is compact;
\item zero-dimensional (called pairwise zero-dimensional in \cite{Reilly}) if $\tau_+$ has a basis of $\tau_-$-closed sets and $\tau_-$ has a basis of $\tau_+$-closed sets.
\end{enumerate}
\end{defn}

For a bitopological space $(X,\tau_+,\tau_-)$, we write $\leq_+$ and $\leq_-$ for the specialization orders of the topological spaces $(X,\tau_+)$ and $(X,\tau_-)$, respectively. Let $\leq$ denote the intersection of $\leq_+$ and the opposite of $\leq_-$; that is,
\[
\leq \;=\; \leq_+ \cap \geq_-.
\]

\begin{defn}\label{order separated} (\cite[Definition 3.7]{JM2008})
A bitopological space $(X, \tau_{+},\tau_{-})$ is \emph{order-separated} provided that the binary relation $\leq\;  =\; \leq_+\cap \geq_-$ is a partial order, and that if $x\not\leq y$, then there exist a $\tau_+$-neighborhood $U$ of $x$ and a $\tau_-$-neighborhood $V$ of $y$ such that $U\cap V=\emptyset$.
\end{defn}

Lemma 3.8 of Jung and Moshier \cite{JM2008} states that in an order-separated bitopological space $(X,\tau_+,\tau_-)$, the specialization order $\leq_+$ of $(X,\tau_+)$ is dual to the specialization order $\leq_-$ of $(X,\tau_-)$; hence
\[
\leq \;=\; \leq_+ \;=\; \geq_-.
\]

\begin{lem}\label{compact opens and dclopens} {\rm(\cite[Lemma 4.1.29]{Jakl})}
Let $(X,\tau_+,\tau_-)$ be an order-separated bitopological space.
\begin{enumerate}[label=(\roman*)] 
\item Every compact saturated subset of $(X,\tau_+)$ is closed in $(X,\tau_-)$. In particular, every compact open subset of $(X,\tau_+)$ is closed in $(X,\tau_-)$.
\item Every compact saturated subset of $(X,\tau_-)$ is closed in $(X,\tau_+)$. In particular, every compact open subset of $(X,\tau_-)$ is closed in $(X,\tau_+)$.
\end{enumerate}
\end{lem}

A bitopological space $(X,\tau_+,\tau_-)$ is \emph{totally order-separated} if the binary relation $\leq$ is a partial order, and whenever $x\not\leq y$, there exists a $\tau_+$-open and $\tau_-$-closed set containing $x$ but not $y$. 

A totally order-separated bitopological space is exactly a totally order-disconnected space in the sense of \cite[Definition 8.12]{JM2006}. The term \emph{totally order-separated} is chosen because, for a topological space $(X,\tau)$, the bitopological space $(X,\tau,\tau)$ is totally order-separated if and only if $(X,\tau)$ is totally separated in the sense of \cite[page 69]{Johnstone}. Clearly, every totally order-separated bitopological space is order-separated.

\begin{prop}{\rm(\cite{YZ2026})} \label{Stone bitopological space}
For a compact and $T_0$ bitopological space $(X,\tau_+,\tau_-)$, the following are equivalent:
\begin{enumerate}[label=(\arabic*)] 
\item $(X,\tau_+,\tau_-)$ is totally order-separated;
\item $(X,\tau_+,\tau_-)$ is zero-dimensional.
\end{enumerate}
\end{prop}

A bitopological space satisfying the equivalent conditions of Proposition \ref{Stone bitopological space} is called a \emph{d-Boolean bitopological space}, or simply a \emph{d-Boolean space}. The category of d-Boolean spaces and continuous maps is denoted by
\[
{\bf dBoolSp}.
\]

As noted in \cite{YZ2026}, d-Boolean spaces are precisely the compact and totally order-disconnected spaces in the sense of \cite{JM2006}, and also the pairwise Stone spaces in the sense of \cite[Definition 2.10]{BBGK}. It is known that the category of d-Boolean spaces is equivalent to the category of spectral spaces and spectral maps:

\begin{thm} {\rm(\cite[Theorem 4.7]{BBGK})}\label{Spec is equivalent to dBoolsp}
For every spectral space $(X,\tau)$, the bitopological space $(X,\tau,\tau^k)$ is d-Boolean. Moreover, the assignment
$ (X,\tau)\mapsto (X,\tau,\tau^k) $
gives rise to an equivalence of categories
\[
\lambda\colon {\bf Spectral} \to {\bf dBoolSp}.
\]
\end{thm}

In this paper, we introduce a subcategory of bitopological spaces, namely the category ${\bf dSpectral}$ of d-spectral bitopological spaces and d-spectral maps. The relationship between d-Boolean spaces and d-spectral spaces is analogous to that between classical Boolean spaces and spectral spaces, yielding the following chain of categorical inclusions:
\[
{\bf BoolSp}\subseteq {\bf Spectral}\cong {\bf dBoolSp}\subseteq {\bf dSpectral}.
\]

It is known that Boolean spaces form a simultaneously reflective and coreflective full subcategory of ${\bf Spectral}$. We extend this fact to the bitopological setting by proving that ${\bf dBoolSp}$ is likewise a simultaneously reflective and coreflective full subcategory of ${\bf dSpectral}$. Consequently, the categories of spectral spaces and d-spectral spaces provide well-behaved extensions of the category of Boolean spaces.

\begin{defn} (\cite{JM2008}) \label{defn of d-sober}
A bitopological space $(X,\tau_+,\tau_-)$ is \emph{d-sober} provided that for every pair $(K_+,K_-)$ of subsets of $X$ satisfying the following conditions:
\begin{enumerate}[label=(\roman*)] 
\item $K_+$ is an irreducible closed set of $(X,\tau_+)$, and $K_-$ is an irreducible closed set of $(X,\tau_-)$;
\item for all $U\in\tau_+$ and $V\in\tau_-$, $U\cap V=\emptyset$ implies either $K_+\cap U=\emptyset$ or $K_-\cap V=\emptyset$;
\item for all $U\in\tau_+$ and $V\in\tau_-$, $U\cup V=X$ implies either $K_+\cap U\neq\emptyset$ or $K_-\cap V\neq\emptyset$,
\end{enumerate}
there exists a unique $x\in X$ such that $K_+=\overline{\{x\}}$ in $(X,\tau_+)$ and $K_-=\overline{\{x\}}$ in $(X,\tau_-)$.
\end{defn}

It is known that every order-separated bitopological space is d-sober \cite[Theorem 4.13]{JM2006}, and every d-sober bitopological space is $T_0$. As remarked in \cite[page 34]{JM2006}, d-sobriety is a subtle constraint on the interaction between the topologies $\tau_+$ and $\tau_-$. For instance, we have:

\begin{lem}{\rm(\cite[Lemma 4.3]{JM2006})}
Let $\tau_+$ and $\tau_-$ be two $T_0$ topologies on a set $X$, and assume that the bitopological space $(X,\tau_+,\tau_-)$ is d-sober. Then the intersection of the specialization orders $\leq_+$ and $\leq_-$ is the identity relation on $X$.
\end{lem}

The above lemma implies that not every finite $T_0$ bitopological space is d-sober.

\begin{defn}\label{d-Spectral space}
A bitopological space $(X,\tau_+,\tau_-)$ is said to be \emph{d-spectral} if:
\begin{enumerate}[label=(\roman*)] 
\item $(X,\tau_+,\tau_-)$ is compact;
\item $(X,\tau_+,\tau_-)$ is d-sober;
\item both $\tau_+$ and $\tau_-$ are coherent frames.
\end{enumerate}
\end{defn}

The following proposition shows that d-Boolean spaces are precisely the order-separated d-spectral spaces.

\begin{prop} \label{d-Boolean via d-spectral}
A bitopological space $(X,\tau_+,\tau_-)$ is d-Boolean if and only if it is d-spectral and order-separated.
\end{prop}

\begin{proof}
For sufficiency, it remains to show that $(X,\tau_+,\tau_-)$ is zero-dimensional. Note that the finite elements of $\tau_+$ are exactly the compact open subsets of $(X,\tau_+)$. Since $\tau_+$ is a coherent frame, the compact open subsets of $(X,\tau_+)$ form a basis for $\tau_+$. As $(X,\tau_+,\tau_-)$ is order-separated, Lemma \ref{compact opens and dclopens} implies that these compact open sets are $\tau_-$-closed. Hence, the $\tau_+$-open and $\tau_-$-closed sets form a basis for $(X,\tau_+)$. Similarly, the $\tau_-$-open and $\tau_+$-closed sets form a basis for $(X,\tau_-)$. Therefore, $(X,\tau_+,\tau_-)$ is zero-dimensional.

For necessity, since $(X,\tau_+,\tau_-)$ is d-Boolean, it is totally order-separated and hence order-separated. It remains to prove that $\tau_+$ and $\tau_-$ are coherent frames. By zero-dimensionality of $(X,\tau_+,\tau_-)$, it suffices to show that the compact open subsets of $(X,\tau_+)$ are precisely the $\tau_+$-open and $\tau_-$-closed sets. By Lemma \ref{compact opens and dclopens}, the compact open subsets of $(X,\tau_+)$ are $\tau_-$-closed. Conversely, compactness of $(X,\tau_+\vee\tau_-)$ implies that any $\tau_+$-open and $\tau_-$-closed set is compact in $(X,\tau_+)$. Thus, the compact open subsets of $(X,\tau_+)$ are precisely the $\tau_+$-open and $\tau_-$-closed sets. Likewise, $\tau_-$ is a coherent frame.
\end{proof}

\begin{prop}
Let $(X,\tau_+,\tau_-)$ be a finite bitopological space. Then:
\begin{enumerate}[label=(\roman*)] 
\item $(X,\tau_+,\tau_-)$ is d-spectral if and only if it is d-sober;
\item $(X,\tau_+,\tau_-)$ is d-Boolean if and only if it is order-separated.
\end{enumerate}
\end{prop}

\begin{proof} Straightforward. \end{proof}

For a bitopological space $(X,\tau_+,\tau_-)$, the bitopological space
\[
(X, \tau_+^k, \tau_-^k)
\]
is called the \emph{de Groot dual} of $(X,\tau_+,\tau_-)$; the bitopological space
\[
(X,\tau_+\vee\tau_-^k,\; \tau_-\vee\tau_+^k)
\]
is called the \emph{patch space} of $(X,\tau_+,\tau_-)$.

Note that for the patch space of $(X,\tau_+,\tau_-)$, the first topology is the join of $\tau_+$ and the co-compact topology of $(X,\tau_-)$, while the second topology is the join of $\tau_-$ and the co-compact topology of $(X,\tau_+)$.

\begin{exmp}
Let $\mathbb{B}$ be the Boolean algebra $\{0,1,{t\!t},{f\!\!f}\}$, where $0$ is the bottom, $1$ is the top, and ${t\!t},{f\!\!f}$ are complements of each other, as shown below:
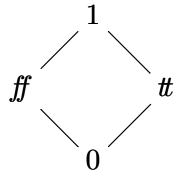
\[
\begin{tikzpicture}[baseline=(one.base)]
  \node (one)  at (0,0)       {$1$};
  \node (ff)   at (-1,-1) {$f\!\!f$};
  \node (tt)   at (1,-1)  {$t\!t$};
  \node (zero) at (0,-2)      {$0$};
  \draw (one) -- (ff) -- (zero) -- (tt) -- (one);
\end{tikzpicture}
\]
Let  
\[
\tau_+=\{\mathbb{B},\emptyset,\{1,{t\!t}\}\}, \quad
\tau_-=\{\mathbb{B},\emptyset,\{1,{f\!\!f}\}\}.
\]
Then $(\mathbb{B},\tau_+,\tau_-)$ is a bitopological space, which we also denote by $\mathbb{B}$. This space is precisely the bitopological space $\mathbb{S}.\mathbb{S}$ of Jung and Moshier \cite[page 38]{JM2006}.
\begin{enumerate}[label=(\roman*)] 
\item The bitopological space $\mathbb{B}$ is d-sober \cite[Example 4.2]{JM2006}, hence d-spectral.
\item The de Groot dual of $\mathbb{B}$ is the space $(\mathbb{B},\tau_+^k,\tau_-^k)$ with
\[
\tau_+^k=\{\mathbb{B},\emptyset,\{0,{f\!\!f}\}\}, \quad
\tau_-^k=\{\mathbb{B},\emptyset,\{0,{t\!t}\}\}.
\]
Clearly, the de Groot dual of $\mathbb{B}$ is homeomorphic to $\mathbb{B}$, hence d-spectral.
\item The patch space of $\mathbb{B}$, denoted by $\mathbb{D}$, has underlying set $\mathbb{B}$, first topology
\[
\{\mathbb{B},\emptyset,\{0,{t\!t}\},\{1,{t\!t}\}, \{0,{t\!t},1\}, \{{t\!t}\}\},
\]
and second topology
\[
\{\mathbb{B},\emptyset,\{0,{f\!\!f}\},\{1,{f\!\!f}\},\{0,{f\!\!f},1\},\{{f\!\!f}\}\}.
\]
The specialization order of the first topology is
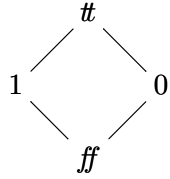
\[
\begin{tikzpicture}[baseline=(one.base)]
  \node (one)  at (-1,-1)       {$1$};
  \node (ff)   at (0,-2) {$f\!\!f$};
  \node (tt)   at (0,0)  {$t\!t$};
  \node (zero) at (1,-1)      {$0$};
  \draw (one) -- (tt) -- (zero) -- (ff) -- (one);
\end{tikzpicture}
\]

It is easily verified that $\mathbb{D}$ is order-separated, hence d-Boolean. The space $\mathbb{D}$ appears in \cite[Remark 5.20]{YZ2026} as a dualizing object for a duality between d-Boolean spaces and d-Boolean algebras.
\end{enumerate}
\end{exmp}

We shall say that a continuous map
\[
f\colon (X,\tau_+,\tau_-) \to (X',\tau_+',\tau_-')
\]
between bitopological spaces is \emph{d-spectral} if both
\[
f\colon (X,\tau_+) \to (X',\tau_+')
\quad\text{and}\quad
f\colon (X,\tau_-) \to (X',\tau_-')
\]
are spectral maps between topological spaces. The category of d-spectral spaces and d-spectral maps is denoted by
\[
{\bf dSpectral}.
\]

\begin{prop}
Every continuous map between d-Boolean spaces is d-spectral.
\end{prop}
\begin{proof}
Let
\[
f\colon (X,\tau_+,\tau_-)\to (X',\tau_+',\tau_-')
\]
be a continuous map between d-Boolean spaces. By Lemma \ref{compact opens and dclopens}, each compact open set $U$ of $(X',\tau_+')$ is $\tau_+'$-open and $\tau_-'$-closed. Hence, $f^{-1}(U)$ is $\tau_+$-open and $\tau_-$-closed. Since $(X,\tau_+\vee\tau_-)$ is compact, $f^{-1}(U)$ is compact in $(X,\tau_+\vee\tau_-)$, and therefore a compact open subset of $(X,\tau_+)$. Thus,
\[
f\colon (X,\tau_+)\to (X',\tau_+')
\]
is a spectral map. Similarly,
\[
f\colon (X,\tau_-)\to (X',\tau_-')
\]
is a spectral map.
\end{proof}

Consequently, the category of d-Boolean spaces and continuous maps is a full subcategory of the category of d-spectral spaces and d-spectral maps.
 
\section{d-lattices and d-Boolean algebras}\label{dLB}

It is well known that spectral spaces are precisely the spectra of distributive lattices. As we shall see, d-spectral bitopological spaces are precisely the d-spectra of d-lattices. This section reviews some basic notions concerning d-lattices, d-Boolean algebras, and d-frames.

Suppose $(L,\sqsubseteq,0,1)$ is a distributive lattice, and ${t\!t},{f\!\!f}$ are a complementary pair of elements of $L$, i.e.,
\[
{t\!t}\sqcup{f\!\!f}=1 \quad\text{and}\quad {t\!t}\sqcap{f\!\!f}=0.
\]
Let $L_+$ denote the lower set $\downarrow\!{t\!t}$ and $L_-$ the lower set $\downarrow\!{f\!\!f}$ of $L$. Then the assignment
\[
a \mapsto (a\sqcap{t\!t},\; a\sqcap{f\!\!f})
\]
is an order isomorphism from $L$ to the product lattice $L_+\times L_-$; its inverse sends each element $(a,b)\in L_+\times L_-$ to $a\sqcup b$ in $(L,\sqsubseteq)$.

Using the isomorphism $L\cong L_+\times L_-$, we define a new order $\leq$ on $L$ by
\[
(a_1,b_1)\leq (a_2,b_2) \;\Longleftrightarrow\; a_1\sqsubseteq a_2 \text{ and } b_1\sqsupseteq b_2.
\]
It is readily verified that $(L,\leq)$ is a distributive lattice with top element ${t\!t}$ and bottom element ${f\!\!f}$. The meet $x\wedge y$ and join $x\vee y$ in $(L,\leq)$ are expressed in terms of the operations of $(L,\sqsubseteq)$ as follows:
\[
x\wedge y = (x\sqcap {f\!\!f})\sqcup (y\sqcap {f\!\!f})\sqcup (x\sqcap y),
\]
\[
x\vee y = (x\sqcap {t\!t})\sqcup (y\sqcap {t\!t})\sqcup (x\sqcap y).
\]
Following Jung and Moshier \cite{JM2006, JM2008}, we call $\sqsubseteq$ the \emph{information order} and $\leq$ the \emph{logic order} of $(L;{t\!t},{f\!\!f})$, respectively.

\begin{defn} (\cite{Klinke,KJM})
A \emph{d-lattice} is a structure
\[
\mathcal{L}=(L;{t\!t},{f\!\!f};{\bf con},{\bf tot}),
\]
where $L$ is a distributive lattice, ${t\!t},{f\!\!f}$ are a complementary pair in $L$, and ${\bf con},{\bf tot}\subseteq L$ (called the \emph{consistency predicate} and the \emph{totality predicate}, respectively) satisfy the following conditions:
\begin{itemize} 
\item ${t\!t},{f\!\!f}\in{\bf con}$;
\item ${t\!t},{f\!\!f}\in{\bf tot}$;
\item ${\bf con}$ is a lower set with respect to the information order $\sqsubseteq$;
\item ${\bf tot}$ is an upper set with respect to $\sqsubseteq$;
\item ${\bf con}$ and ${\bf tot}$ are sublattices of $(L,\leq,\wedge,\vee)$, i.e., sublattices under the logic order;
\item ({\bf con}--{\bf tot}) if $\alpha\in{\bf con}$, $\beta\in{\bf tot}$, and
\[
(\alpha\sqcap{t\!t}=\beta\sqcap{t\!t}) \quad\text{or}\quad (\alpha\sqcap{f\!\!f}=\beta\sqcap{f\!\!f}),
\]
then $\alpha\sqsubseteq\beta$.
\end{itemize}
\end{defn}

Let $\mathcal{L}=(L;{t\!t},{f\!\!f};{\bf con},{\bf tot})$ and $\mathcal{M}=(M;{t\!t}_M,{f\!\!f}_M;{\bf con}_M,{\bf tot}_M)$ be d-lattices. A \emph{d-lattice homomorphism}
\[
f\colon \mathcal{L}\to \mathcal{M}
\]
is a lattice homomorphism $f\colon L\to M$ that preserves ${t\!t}$, ${f\!\!f}$, ${\bf con}$, and ${\bf tot}$ in the sense that
\[
f({t\!t})={t\!t}_M,\quad f({f\!\!f})={f\!\!f}_M,\quad f({\bf con})\subseteq {\bf con}_M,\quad f({\bf tot})\subseteq {\bf tot}_M.
\]
The category of d-lattices and d-lattice homomorphisms is denoted by ${\bf dLat}$.

For each d-lattice $\mathcal{L}=(L;{t\!t},{f\!\!f};{\bf con},{\bf tot})$, the underlying lattice $L$ is isomorphic to the product $L_+\times L_-$, where $L_+=\;\downarrow\!{t\!t}$ and $L_-=\;\downarrow\!{f\!\!f}$. Thus, a d-lattice can be equivalently presented as a pair $(L_+,L_-)$ of distributive lattices together with two subsets of $L_+\times L_-$ subject to suitable conditions, as in \cite[Definition 1]{KJM} and \cite[Definition 2.1.2]{Klinke}. This alternative presentation is particularly convenient for constructing d-lattices. Both presentations will be used in this paper. To avoid confusion, we identify $L$ with $L_+\times L_-$, and in particular we identify $L_+$ and $L_-$ with the subsets
\[
\{(a,0):a\sqsubseteq{t\!t}\}\quad\text{and}\quad \{(0,b):b\sqsubseteq{f\!\!f}\}
\]
of the product, respectively.

\begin{center}
\begin{tikzpicture}[scale=0.6]  
\fill [lightgray] (0,0) rectangle (5,4); 
\draw[thick] (0,0) -- (5,0); 
\draw[thick] (0,0) -- (0,4);  
\node at (2.5,-0.4) {{\small $L_+$}}; 
\node at (5,-0.4) {{\small ${t\!t}$}}; 
\node at (-0.5,2) {{\small $L_-$}};  
\node at (-0.4,4) {{\small ${f\!\!f}$}}; 
\node at (-0.3,-0.3) {{\small $0$}}; 
\node at (5.3,4) {{\small $1$}}; 
\draw [fill] (0,4) circle [radius=0.05]; 
\draw [fill] (5,0) circle [radius=0.05]; 
\node at (2.5,-1.5) {{\small $L\cong L_+\times L_-$}};
\end{tikzpicture}
\end{center}

\begin{rem}
Suppose that $\mathcal{L}=(L;{t\!t},{f\!\!f};{\bf con},{\bf tot})$ is a d-lattice with $\{{t\!t},{f\!\!f}\}=\{1,0\}$. Then either $L_+$ or $L_-$ is a singleton, which implies that ${\bf con}=L={\bf tot}$. By condition ({\bf con}--{\bf tot}), one readily deduces that $L$ is a singleton. In other words, up to isomorphism, there is only one d-lattice for which $\{{t\!t},{f\!\!f}\}=\{1,0\}$, namely the terminal object of the category ${\bf dLat}$.
\end{rem}

\begin{defn}
A \emph{d-frame} is a d-lattice $(L;{t\!t},{f\!\!f};{\bf con},{\bf tot})$ such that $L$ is a frame and the consistency predicate ${\bf con}$ is Scott closed in $(L,\sqsubseteq)$.
\end{defn}

In \cite{JM2006,JM2008}, a d-frame is defined as a d-lattice whose underlying lattice is a frame. However, following \cite{Jakl,JJP}, we reserve the term \emph{d-frame} for those d-lattices for which $L$ is a frame and the consistency predicate ${\bf con}$ is Scott closed under the information order.

Let $\mathcal{L}=(L;{t\!t},{f\!\!f};{\bf con},{\bf tot})$ and $\mathcal{M}=(M;{t\!t}_M,{f\!\!f}_M;{\bf con}_M,{\bf tot}_M)$ be d-frames. A \emph{d-frame homomorphism}
\[
g\colon \mathcal{L}\to \mathcal{M}
\]
is a frame homomorphism $g\colon L\to M$ that preserves ${t\!t}$, ${f\!\!f}$, ${\bf con}$, and ${\bf tot}$ \cite{JM2006,JM2008}. The category of d-frames and d-frame homomorphisms is denoted by ${\bf dFrm}$.

The Boolean algebra $\mathbb{B}=\{0,1,{t\!t},{f\!\!f}\}$ admits a unique d-frame structure, namely
\[
{\bf con}=\{0,{t\!t},{f\!\!f}\},\qquad {\bf tot}=\{1,{t\!t},{f\!\!f}\}.
\]
Throughout this paper, $\mathbb{B}$ is always assumed to be equipped with this d-frame structure.

\begin{exmp} \label{d-frame of subsets}
Let $X$ be a set and let $\mathcal{P}(X)$ be its power set. Then the structure
\[
(\mathcal{P}(X)\times\mathcal{P}(X);{t\!t},{f\!\!f};{\bf con},{\bf tot})
\]
is a d-lattice, indeed a d-frame, called the \emph{d-frame of subsets} of $X$. Here
\begin{itemize}  
\item ${t\!t}=(X,\emptyset)$, ${f\!\!f}=(\emptyset,X)$;
\item ${\bf con}=\{(U,V)\in \mathcal{P}(X)\times\mathcal{P}(X) : U\cap V=\emptyset\}$;
\item ${\bf tot}=\{(U,V)\in \mathcal{P}(X)\times\mathcal{P}(X) : U\cup V=X\}$.
\end{itemize}
The d-frame of all subsets of $X$ is an instance of the construction of d-frames given on page 47 of \cite{JM2006}, applied to the power set $\mathcal{P}(X)$.
\end{exmp}

Suppose that $\mathcal{L}=(L;{t\!t},{f\!\!f};{\bf con}_L,{\bf tot}_L)$ is a d-lattice and that $N_+$ and $N_-$ are sublattices of $L_+$ and $L_-$, respectively. In particular, $0,{t\!t}\in N_+$ and $0,{f\!\!f}\in N_-$. Let
\[
N= N_+\times N_-, \quad {\bf con}_N= N\cap {\bf con}_L, \quad {\bf tot}_N= N\cap {\bf tot}_L.
\]
Then $\mathcal{N} =(N;{t\!t},{f\!\!f};{\bf con}_N,{\bf tot}_N)$ is a d-lattice, called a \emph{sub-d-lattice} of $\mathcal{L}$. The notion of \emph{sub-d-frame} is defined analogously.

\begin{exmp} (\cite{JM2006,JM2008}) \label{the functor dO}
For a bitopological space $(X,\tau_+,\tau_-)$, the structure
\[
\mathrm{d}\mathcal{O}(X,\tau_+,\tau_-)\coloneqq (L;{t\!t},{f\!\!f};{\bf con},{\bf tot})
\]
is a d-frame, indeed a sub-d-frame of the d-frame of subsets of $X$. Here
\begin{itemize} 
\item $L=\tau_+\times\tau_-$;
\item ${t\!t}=(X,\emptyset)$, ${f\!\!f}=(\emptyset,X)$;
\item ${\bf con}=\{(U,V)\in \tau_+\times\tau_- : U\cap V=\emptyset\}$;
\item ${\bf tot}=\{(U,V)\in \tau_+\times\tau_- : U\cup V=X\}$.
\end{itemize}
This d-frame is called the \emph{d-frame of open sets} of $(X,\tau_+,\tau_-)$. The assignment
\[
(X,\tau_+,\tau_-) \mapsto \mathrm{d}\mathcal{O}(X,\tau_+,\tau_-)
\]
gives a contravariant functor
\[
\mathrm{d}\mathcal{O}\colon {\bf BiTop}\to  {\bf dFrm}^{\rm op}.
\]
\end{exmp}

\begin{exmp}
For a d-spectral space $(X,\tau_+,\tau_-)$, let $K_+$ denote the sublattice of $\tau_+$ consisting of the compact open subsets of $(X,\tau_+)$, and let $K_-$ denote the analogous sublattice of $\tau_-$. Then the structure
\[
\mathrm{K}\mathcal{O}(X,\tau_+,\tau_-)\coloneqq (L;{t\!t},{f\!\!f};{\bf con},{\bf tot})
\]
is a d-lattice, indeed a sub-d-lattice of the d-frame $\mathrm{d}\mathcal{O}(X,\tau_+,\tau_-)$. Here
\begin{itemize} 
\item $L=K_+\times K_-$;
\item ${t\!t}=(X,\emptyset)$, ${f\!\!f}=(\emptyset,X)$;
\item ${\bf con}=\{(U,V)\in K_+\times K_- : U\cap V=\emptyset\}$;
\item ${\bf tot}=\{(U,V)\in K_+\times K_- : U\cup V=X\}$.
\end{itemize}
We call $\mathrm{K}\mathcal{O}(X,\tau_+,\tau_-)$ the \emph{d-lattice of compact open sets} of $(X,\tau_+,\tau_-)$. The assignment
\[
(X,\tau_+,\tau_-) \mapsto \mathrm{K}\mathcal{O}(X,\tau_+,\tau_-)
\]
gives a contravariant functor
\[
\mathrm{K}\mathcal{O}\colon {\bf dSpectral}\to  {\bf dLat}^{\rm op}.
\] 
\end{exmp}

The category ${\bf DisLat}$ of distributive lattices can be embedded into ${\bf dLat}$ in natural ways. One such embedding is the functor
\[
\lambda\colon {\bf DisLat}\to {\bf dLat}
\]
introduced in \cite{YZ2026}, which sends a distributive lattice $L$ to the d-lattice
\[
\lambda(L)\coloneqq (L\times L^{\rm op};{t\!t},{f\!\!f};{\bf con},{\bf tot}),
\]
where
\begin{itemize} 
\item ${t\!t}=(1,1)$, ${f\!\!f}=(0,0)$;
\item ${\bf con}=\{(a,b)\in L\times L : a\sqsubseteq b\}$;
\item ${\bf tot}=\{(a,b)\in L\times L : a\sqsupseteq b\}$.
\end{itemize}
A schematic representation of the consistency and totality predicates in $\lambda(L)$ is given below:
\begin{center} 
\begin{tikzpicture}
\fill[lightgray] (0,0) -- (3,0) -- (3,3) -- cycle;
\draw[thick] (0,0) -- (3,0) -- (3,3) -- cycle;
\draw[thick] (0,0) -- (0,3) -- (3,3) -- cycle;
\node at (2,1) {{\bf tot}};
\node at (1,2) {{\bf con}};
\node at (3.2,3.1) {${t\!t}$};
\node at (-0.3,-0.1) {${f\!\!f}$};
\end{tikzpicture}
\end{center} 

\begin{defn}(\cite[Definition 2.2.2]{Klinke})
Let $\mathcal{L}=(L;{t\!t},{f\!\!f};{\bf con},{\bf tot})$ be a d-lattice and $x\in L$. We say that $x$ is \emph{d-complemented} if:
\begin{itemize} 
\item either $x\in L_+$ and there exists a (necessarily unique) $x^\dag\in L_-$ such that
\[
x\sqcup x^\dag\in {\bf con}\cap{\bf tot};
\]
\item or $x\in L_-$ and there exists a (necessarily unique) $x^\dag\in L_+$ such that
\[
x^\dag\sqcup x\in {\bf con}\cap{\bf tot}.
\]
\end{itemize}
In either case, $x^\dag$ is called a \emph{d-complement} of $x$.
\end{defn}

\begin{defn} \label{defn of d-boolean} (\cite{YZ2026})
A \emph{d-Boolean algebra} is a d-lattice $(L;{t\!t},{f\!\!f};{\bf con},{\bf tot})$ in which every element of $L_+$ and every element of $L_-$ is d-complemented. The full subcategory of ${\bf dLat}$ consisting of d-Boolean algebras is denoted by ${\bf dBA}$.
\end{defn}

It is clear that for each set $X$, the d-lattice of subsets of $X$ (Example \ref{d-frame of subsets}) is a d-Boolean algebra.

\begin{exmp}
Let $(X,\tau_+,\tau_-)$ be a bitopological space. Define
\[
L_+=\{U\in\tau_+ : X\setminus U\in\tau_-\},\quad
L_-=\{V\in\tau_- : X\setminus V\in\tau_+\}.
\]
Equivalently, $L_+$ consists of the $\tau_+$-open and $\tau_-$-closed subsets of $X$, and $L_-$ consists of the $\tau_-$-open and $\tau_+$-closed subsets. Then the structure
\[
(L;{t\!t},{f\!\!f};{\bf con},{\bf tot})
\]
is a d-Boolean algebra, where
\begin{itemize} 
\item $L=L_+\times L_-$;
\item ${t\!t}=(X,\emptyset)$, ${f\!\!f}=(\emptyset,X)$;
\item ${\bf con}=\{(U,V)\in L_+\times L_- : U\cap V=\emptyset\}$;
\item ${\bf tot}=\{(U,V)\in L_+\times L_- : U\cup V=X\}$.
\end{itemize}
This d-Boolean algebra is called the \emph{d-Boolean algebra of d-clopen sets} of $(X,\tau_+,\tau_-)$; it is a sub-d-lattice of the d-frame of open sets of $(X,\tau_+,\tau_-)$. The assignment
\[
(X,\tau_+,\tau_-) \mapsto \operatorname{dClop}(X,\tau_+,\tau_-)
\]
gives a contravariant functor
\[
\operatorname{dClop}\colon {\bf BiTop}\to  {\bf dBA}^{\rm op}.
\]
\end{exmp}

The d-Boolean algebra of d-clopen sets is an instance of a general construction. Let $\mathcal{L}=(L;{t\!t},{f\!\!f};{\bf con},{\bf tot})$ be a d-lattice. Define
\[
B_+=\{a\in L_+ : a \text{ is d-complemented}\},\quad
B_-=\{b\in L_- : b \text{ is d-complemented}\}.
\]
Then $B_+$ is a sublattice of $L_+$, and $B_-$ is a sublattice of $L_-$. Let
\[
\operatorname{dB} L=\{a\sqcup b : a\in B_+,\; b\in B_-\}.
\]
Then $\operatorname{dB} L$ is a sublattice of $L$ containing ${t\!t}$ and ${f\!\!f}$. The structure
\[
\operatorname{dB}\mathcal{L}\coloneqq (\operatorname{dB} L;{t\!t},{f\!\!f};{\bf con}_{\operatorname{dB}},{\bf tot}_{\operatorname{dB}})
\]
is a d-Boolean algebra, where
\[
{\bf con}_{\operatorname{dB}}={\bf con}\cap \operatorname{dB} L,\qquad {\bf tot}_{\operatorname{dB}}={\bf tot}\cap \operatorname{dB} L.
\]
The assignment $\mathcal{L}\mapsto \operatorname{dB}\mathcal{L}$ gives a functor
\[
\operatorname{dB}\colon {\bf dLat}\to  {\bf dBA}
\]
that is right adjoint to the inclusion ${\bf dBA}\to {\bf dLat}$; see \cite[Proposition 4.18]{YZ2026}. Thus, ${\bf dBA}$ is a coreflective full subcategory of ${\bf dLat}$. It is clear that $\operatorname{dClop} = \operatorname{dB}\circ\mathrm{d}\mathcal{O}$.

\begin{prop}{\rm(\cite[Proposition 4.17]{YZ2026})}
For every distributive lattice $L$, the d-lattice $\lambda(L)$ is a d-Boolean algebra. Restricting the codomain of the functor  $\lambda$ to ${\bf dBA}$ yields an equivalence of categories
\[
\lambda\colon {\bf DisLat}\to  {\bf dBA}.
\]
\end{prop}

\begin{defn}
Let $\mathcal{L}=(L;{t\!t},{f\!\!f};{\bf con},{\bf tot})$ be a d-frame.
\begin{enumerate}[label=(\roman*)] 
\item (\cite[Definition 7.5]{JM2008}) $\mathcal{L}$ is \emph{compact} if ${\bf tot}$ is a Scott open subset of $(L,\sqsubseteq)$.
\item (\cite[Definition 2.3.6]{Jakl}) $\mathcal{L}$ is \emph{d-zero-dimensional} if every element of $L$ is a join of d-complemented elements.
\end{enumerate}
\end{defn}

For a bitopological space $(X,\tau_+,\tau_-)$, the d-frame $\mathrm{d}\mathcal{O}(X,\tau_+,\tau_-)$ is compact and d-zero-dimensional if and only if $(X,\tau_+,\tau_-)$ is compact and zero-dimensional in the sense of Definition \ref{compact bitop}; see \cite[page 19]{Jakl}.

For a d-lattice $\mathcal{L}=(L;{t\!t},{f\!\!f};{\bf con},{\bf tot})$, define
\[
{\bf con}_{\operatorname{Idl}}=\{I\in\operatorname{Idl} L : I\subseteq {\bf con}\},\quad
{\bf tot}_{\operatorname{Idl}}=\{I\in\operatorname{Idl} L : I\cap {\bf tot}\neq\emptyset\}.
\]
Then
\[
\operatorname{Idl}\mathcal{L}\coloneqq (\operatorname{Idl} L;\downarrow\!{t\!t},\downarrow\!{f\!\!f};{\bf con}_{\operatorname{Idl}},{\bf tot}_{\operatorname{Idl}})
\]
is a d-frame, called the \emph{d-frame of ideals} of $\mathcal{L}$. If $f\colon\mathcal{L}\to\mathcal{M}$ is a d-lattice homomorphism, then the map
\[
\operatorname{Idl} L \to \operatorname{Idl} M,\qquad I \mapsto f(I)=\{b\in M : b\sqsubseteq f(a)\text{ for some }a\in I\}
\]
is a d-frame homomorphism $\operatorname{Idl}\mathcal{L}\to\operatorname{Idl}\mathcal{M}$. Thus we obtain a functor
\[
\operatorname{Idl}\colon {\bf dLat}\to  {\bf dFrm}.
\] 

\begin{prop}{\rm(\cite[Proposition 4.21]{YZ2026})} \label{the d-Boolean algebra coreflection of Idl}
Every d-Boolean algebra $\mathcal{L}=(L;{t\!t},{f\!\!f};{\bf con},{\bf tot})$ is isomorphic to $\operatorname{dB}\circ\operatorname{Idl}\mathcal{L}$.
\end{prop}

\section{The d-spectrum  of a d-lattice}\label{Spectra}

Let $\mathcal{L}=(L;{t\!t},{f\!\!f};{\bf con},{\bf tot})$ be a d-lattice. A \emph{d-filter} of $\mathcal{L}$ is a map $f\colon L\to \mathbb{B}$ such that
\begin{enumerate}[label=(\roman*)] 
\item $f({t\!t})\sqsupseteq {t\!t}$ and $f({f\!\!f})\sqsupseteq {f\!\!f}$;
\item $f$ preserves ${\bf tot}$ and finite meets; that is,
\[
f({\bf tot})\subseteq \{1,{t\!t},{f\!\!f}\}, \qquad
f(x\sqcap y)=f(x)\sqcap f(y) \text{ for all }x,y\in L.
\]
\end{enumerate}

A d-filter $f$ is said to be \emph{proper} if $f(0)=0$. Furthermore, a proper d-filter $f$ is \emph{prime} if it also preserves ${\bf con}$ and finite joins; that is,
\[
f({\bf con})\subseteq \{0,{t\!t},{f\!\!f}\}, \qquad
f(x\sqcup y)=f(x)\sqcup f(y) \text{ for all }x,y\in L.
\]

A prime d-filter of $\mathcal{L}$ is exactly a d-lattice homomorphism $f\colon \mathcal{L}\to \mathbb{B}$. The following characterization of prime d-filters of d-Boolean algebras is useful.

\begin{lem}\label{prime d-filter} {\rm(\cite[Lemma 5.12]{YZ2026})}
Let $f\colon L\to\mathbb{B}$ be a proper d-filter of a d-Boolean algebra $(L;{t\!t},{f\!\!f};{\bf con},{\bf tot})$. Then $f$ is prime if and only if it satisfies:
\begin{enumerate}[label=(\roman*)] 
\item for all $a\in L_+$, $f(a)={t\!t} \iff f(a^\dag)=0$;
\item for all $b\in L_-$, $f(b)={f\!\!f} \iff f(b^\dag)=0$.
\end{enumerate}
\end{lem}

We write
\[
\operatorname{dSpec}\mathcal{L}
\]
for the set of all prime d-filters of $\mathcal{L}$, and we endow $\operatorname{dSpec}\mathcal{L}$ with a bitopological structure by taking as $\tau_+$ the topology generated by the collection
\[
\phi_+(a)=\{g\in\operatorname{dSpec}\mathcal{L}: g(a)={t\!t}\},\quad a\in L_+,
\]
and as $\tau_-$ the topology generated by
\[
\phi_-(b)=\{g\in\operatorname{dSpec}\mathcal{L}: g(b)={f\!\!f}\},\quad b\in L_-.
\]
The bitopological space
\[
(\operatorname{dSpec}\mathcal{L},\tau_+,\tau_-)
\]
 is called the \emph{d-spectrum} of the d-lattice $\mathcal{L}$. 

It is readily verified that the sets $\phi_+(a)$, $a\in L_+$, are compact open sets of $\tau_+$, and that every open set of $\tau_+$ is of the form
\[
\phi_+(I_+)\coloneqq \{g\in\operatorname{dSpec}\mathcal{L}: \exists a\in I_+,\; g(a)={t\!t}\}
\]
for some ideal $I_+$ of the distributive lattice $L_+$; analogously for $\tau_-$.

Before proceeding to prove that the d-spectrum of every d-lattice is a d-spectral space, we prove a lemma analogous to Lemma \ref{spectrum preserves involution}.

Let $\mathcal{L}=(L;{t\!t},{f\!\!f};{\bf con},{\bf tot})$ be a d-lattice. The \emph{order-dual} $\mathcal{L}^\partial$ of $\mathcal{L}$ \cite[Definition 2.1.4]{Klinke} is the d-lattice
\[
(L^{\rm op};{t\!t}^\partial,{f\!\!f}^\partial;{\bf con}^\partial,{\bf tot}^\partial),
\]
where $L^{\rm op}$ is the opposite lattice of $L$, and
\begin{itemize} 
\item ${t\!t}^\partial={f\!\!f}$,
\item ${f\!\!f}^\partial={t\!t}$,
\item ${\bf con}^\partial={\bf tot}$,
\item ${\bf tot}^\partial={\bf con}$.
\end{itemize}

For the d-lattice $\mathcal{L}^\partial$, the subset $L_+^\partial$ is $\{a\sqcup{f\!\!f}: a\in L_+\}$ with the opposite order, and $L_-^\partial$ is $\{{t\!t}\sqcup b: b\in L_-\}$ with the opposite order, as visualized below:
\begin{center}
\begin{tikzpicture}[scale=0.6]  
\fill [lightgray] (0,0) rectangle (5,4); 
\draw[thick] (0,4) -- (5,4); 
\draw[thick] (5,0) -- (5,4);  
\node at (2.5,4.5) {{\small $L_+^\partial$}}; 
\node at (6,-0.3) {{\small ${f\!\!f}^\partial={t\!t}$}}; 
\node at (5.6,2) {{\small $L_-^\partial$}};  
\node at (-1,4.3) {{\small ${f\!\!f}={t\!t}^\partial$}}; 
\node at (-1,-0.3) {{\small $0=1^\partial$}}; 
\node at (6,4.3) {{\small $0^\partial=1$}}; 
\draw [fill] (0,4) circle [radius=0.05]; 
\draw [fill] (5,0) circle [radius=0.05];  
\end{tikzpicture}
\end{center}

\begin{lem}\label{dspec preserves involution}
For every d-lattice $\mathcal{L}=(L;{t\!t},{f\!\!f};{\bf con},{\bf tot})$, the d-spectrum of the order-dual $\mathcal{L}^\partial$ is homeomorphic to the de Groot dual of the d-spectrum of $\mathcal{L}$.
\end{lem}
\begin{proof} 
Let $(\operatorname{dSpec}\mathcal{L},\tau_+,\tau_-)$ be the d-spectrum of $\mathcal{L}$. Let $\neg\colon \mathbb{B}\to\mathbb{B}$ be the negation operator on the Boolean algebra $\mathbb{B}$; that is,
\[
\neg 0=1,\quad \neg 1=0,\quad \neg{t\!t}={f\!\!f},\quad \neg{f\!\!f}={t\!t}.
\]
It is easily seen that $g\colon L\to\mathbb{B}$ is a prime d-filter of $\mathcal{L}$ if and only if $\neg g$ is a prime d-filter of $\mathcal{L}^\partial$. Thus the correspondence $g\mapsto \neg g$ is a bijection from the set of prime d-filters of $\mathcal{L}$ to that of $\mathcal{L}^\partial$. Hence we may view the d-spectrum of $\mathcal{L}^\partial$ as the set $\operatorname{dSpec}\mathcal{L}$ equipped with topologies $\tau_+^\partial$ and $\tau_-^\partial$, where $\tau_+^\partial$ is generated by the basis
\[
\phi^\partial_+(a)\coloneqq
\{g\in\operatorname{dSpec}\mathcal{L}: \neg g(a\sqcup{f\!\!f})={t\!t}\}
=
\{g\in\operatorname{dSpec}\mathcal{L}: g(a)=0\},\quad a\in L_+,
\]
and $\tau_-^\partial$ is generated by
\[
\phi^\partial_-(b)\coloneqq
\{g\in\operatorname{dSpec}\mathcal{L}: \neg g({t\!t}\sqcup b)={f\!\!f}\}
=
\{g\in\operatorname{dSpec}\mathcal{L}: g(b)=0\},\quad b\in L_-.
\]

For every $a\in L_+$ and $g\in\operatorname{dSpec}\mathcal{L}$,
\[
g\in\operatorname{dSpec}\mathcal{L}\setminus\phi_+(a)
\iff g(a)=0
\iff g\in\phi^\partial_+(a).
\]
Then, since the family $\{\phi_+(a):a\in L_+\}$ of compact open sets of $(\operatorname{dSpec}\mathcal{L},\tau_+)$ is a basis for $\tau_+$, it follows that $\tau^\partial_+=\tau_+^k$. Likewise, $\tau^\partial_-=\tau_-^k$. Therefore, the d-spectrum of $\mathcal{L}^\partial$ is homeomorphic to the de Groot dual of the d-spectrum of $\mathcal{L}$.
\end{proof}

In what follows, we show that the d-spectrum of every d-lattice is a d-spectral space. To this end, we first reformulate the notion of d-sobriety in Definition \ref{defn of d-sober}.

As in the classical setting, the functor
\[
\mathrm{d}\mathcal{O}\colon {\bf BiTop}\to  {\bf dFrm}^{\rm op}
\]
from Example \ref{the functor dO} has a right adjoint
\[
\operatorname{dpt} \colon {\bf dFrm}^{\rm op}\to  {\bf BiTop},
\]
which sends each d-frame $\mathcal{L}$ to the bitopological space of its d-points (a d-point of $\mathcal{L}$ is exactly a d-frame homomorphism $\mathcal{L}\to\mathbb{B}$).

The notion of d-sobriety can be reformulated via the adjunction $\mathrm{d}\mathcal{O}\dashv \operatorname{dpt}$: a bitopological space $X$ is d-sober if and only if the unit map $\eta_X\colon X\to \operatorname{dpt} \circ \mathrm{d}\mathcal{O}(X)$ is a homeomorphism. A d-frame $\mathcal{L}$ is called \emph{spatial} if the counit map $\epsilon_\mathcal{L}\colon \mathcal{L} \to \mathrm{d}\mathcal{O}\circ\operatorname{dpt}\mathcal{L}$ is an isomorphism of d-frames. The adjunction $\mathrm{d}\mathcal{O}\dashv\operatorname{dpt}$ then restricts to a dual equivalence between the category of d-sober bitopological spaces and the category of spatial d-frames; see \cite{JM2006,JM2008} for details.

\begin{prop}{\rm(\cite[Proposition 5.15]{YZ2026})}
\label{spectrum as composite}
The d-spectrum of a d-lattice
$\mathcal L=(L;{t\!t},{f\!\!f};{\bf con},{\bf tot})$
is homeomorphic to the bitopological space of
d-points of the d-frame of ideals of $\mathcal{L}$; that is,
\[
\operatorname{dSpec}\mathcal L
\cong
\operatorname{dpt}\circ\operatorname{Idl}\mathcal L.
\]
In particular, $(\operatorname{dSpec}\mathcal L,\tau_+,\tau_-)$
is d-sober.
\end{prop}

\begin{thm}\label{functor dSpec}
For every d-lattice $\mathcal{L}=(L;{t\!t},{f\!\!f};{\bf con},{\bf tot})$, its d-spectrum $(\operatorname{dSpec}\mathcal{L},\tau_+,\tau_-)$ is a d-spectral bitopological space.
\end{thm}

\begin{proof}
Proposition \ref{spectrum as composite} shows that $(\operatorname{dSpec}\mathcal{L},\tau_+,\tau_-)$ is d-sober. It remains to prove that both $\tau_+$ and $\tau_-$ are coherent frames and that the bitopological space $(\operatorname{dSpec}\mathcal{L},\tau_+,\tau_-)$ is compact. These will be established in Proposition \ref{varphi is injective} and Proposition \ref{spectrum is compact} below.
\end{proof}

\begin{prop}\label{varphi is injective}
Let $(\operatorname{dSpec}\mathcal{L},\tau_+,\tau_-)$ be the d-spectrum of a d-lattice $\mathcal{L}=(L;{t\!t},{f\!\!f};{\bf con},{\bf tot})$.
\begin{enumerate}[label=(\roman*)] 
\item The map
\[
\phi_+\colon \operatorname{Idl} L_+ \to \tau_+,\quad
I \mapsto \{g\in\operatorname{dSpec}\mathcal{L}: \exists a\in I,\; g(a)={t\!t}\}
\]
is an isomorphism of frames.
\item The map
\[
\phi_-\colon \operatorname{Idl} L_- \to \tau_-,\quad
J \mapsto \{g\in\operatorname{dSpec}\mathcal{L}: \exists b\in J,\; g(b)={f\!\!f}\}
\]
is an isomorphism of frames.
\end{enumerate}
Hence both $\tau_+$ and $\tau_-$ are coherent frames. Furthermore, the compact open sets of $(\operatorname{dSpec}\mathcal{L},\tau_+)$ are precisely the sets $\phi_+(a)$ with $a\in L_+$; analogously for $(\operatorname{dSpec}\mathcal{L},\tau_-)$.
\end{prop}

\begin{proof}
(i) Clearly, $\phi_+\colon \operatorname{Idl} L_+ \to \tau_+$ is a surjective frame homomorphism. We only need to show that if $I$ and $I'$ are distinct ideals of $L_+$, then $\phi_+(I)\neq \phi_+(I')$. Without loss of generality, assume $I\not\subseteq I'$. Then there exists a prime filter $F_+$ of $L_+$ that intersects $I$ but is disjoint from $I'$.

Define
\begin{align*}
C_F&=\{y\in L_-: x\sqcup y\in{\bf con}\text{ for some }x\in F_+\},\\
T_F&=\{y\in L_-: x\sqcup y\in{\bf tot}\text{ for some }x\in L_+\setminus F_+\}.
\end{align*}
We claim that $C_F$ is an ideal of $L_-$, $T_F$ is a filter of $L_-$, and $C_F\cap T_F=\emptyset$.

It is clear that $C_F$ is a lower set. If $y,y'\in C_F$, then there exist $x,x'\in F_+$ such that $x\sqcup y\in{\bf con}$ and $x'\sqcup y'\in{\bf con}$. Since ${\bf con}$ is a sublattice of $(L,\leq)$, we have
\[
(x\sqcup y)\wedge(x'\sqcup y') = (x\sqcap x')\sqcup (y\sqcup y')\in{\bf con}.
\]
Because $F_+$ is a filter, $x\sqcap x'\in F_+$, so $y\sqcup y'\in C_F$. Hence $C_F$ is an ideal. Similarly, $T_F$ is a filter. To see $C_F\cap T_F=\emptyset$, suppose $y\in C_F\cap T_F$. Then there exist $x\in F_+$ and $x'\in L_+\setminus F_+$ with $x\sqcup y\in{\bf con}$ and $x'\sqcup y\in{\bf tot}$. By the ({\bf con-tot}) axiom, we get $x\sqsubseteq x'$, contradicting that $F_+$ is an upper set.

Since $C_F\cap T_F=\emptyset$, there exists a prime filter $F_-$ of $L_-$ containing $T_F$ and disjoint from $C_F$. Using the pair $(F_+,F_-)$, we construct a prime d-filter $g\in\operatorname{dSpec}\mathcal{L}$ as follows:
\[
g(a\sqcup b)=
\begin{cases}
1, & a\in F_+,\; b\in F_-,\\
{t\!t}, & a\in F_+,\; b\notin F_-,\\
{f\!\!f}, & a\notin F_+,\; b\in F_-,\\
0, & a\notin F_+,\; b\notin F_-,
\end{cases}
\]
for all $a\in L_+$, $b\in L_-$. One verifies that this defines a prime d-filter. Since $F_+\cap I\neq\emptyset$, we have $g\in\phi_+(I)$; since $F_+\cap I'=\emptyset$, we have $g\notin\phi_+(I')$. Thus $\phi_+(I)\neq\phi_+(I')$.

(ii) The proof is analogous to that of (i).
\end{proof}

To prove that the d-spectrum of every d-lattice is a compact bitopological space, we first present some results about d-Boolean algebras: their d-spectra and their relationship to d-lattices.

\begin{prop}\label{spectra of d-Boolean algebras}
Let $\mathcal{L}=(L;{t\!t},{f\!\!f};{\bf con},{\bf tot})$ be a d-Boolean algebra and let $(\operatorname{dSpec}\mathcal{L},\tau_+,\tau_-)$ be its d-spectrum.
\begin{enumerate}[label=(\roman*)] 
\item The d-frame $\operatorname{Idl}\mathcal{L}$ is isomorphic to the d-frame of open sets of $(\operatorname{dSpec}\mathcal{L},\tau_+,\tau_-)$; that is,
\[
\mathrm{d}\mathcal{O}\circ\operatorname{dSpec}\mathcal{L}\cong \operatorname{Idl}\mathcal{L}.
\]
In particular, $\operatorname{Idl}\mathcal{L}$ is spatial.
\item $\mathcal{L}$ is isomorphic to the d-Boolean algebra of d-clopen sets of $(\operatorname{dSpec}\mathcal{L},\tau_+,\tau_-)$; that is,
\[
\operatorname{dClop}\circ\operatorname{dSpec}\mathcal{L}\cong \mathcal{L}.
\]
\item The bitopological space $(\operatorname{dSpec}\mathcal{L},\tau_+,\tau_-)$ is d-Boolean.
\end{enumerate}
\end{prop}

\begin{proof}
(i) This is precisely \cite[Proposition 5.17]{YZ2026}.

(ii) By Proposition \ref{the d-Boolean algebra coreflection of Idl}, $\mathcal{L}\cong \operatorname{dB}\circ\operatorname{Idl}\mathcal{L}$. Hence
\[
\mathcal{L}\cong \operatorname{dB}\circ\operatorname{Idl}\mathcal{L}
\cong \operatorname{dB}\circ\mathrm{d}\mathcal{O}\circ\operatorname{dSpec}\mathcal{L}
\cong \operatorname{dClop}\circ\operatorname{dSpec}\mathcal{L}.
\]

(iii) This follows from \cite[Theorem 5.19]{YZ2026}.
\end{proof}

\begin{prop}\label{spectrum of d-boolean algebra is d-Boolean}
Let $\mathcal{L}=(L;{t\!t},{f\!\!f};{\bf con},{\bf tot})$ be a d-Boolean algebra, and let $(\operatorname{dSpec}\mathcal{L},\tau_+,\tau_-)$ be its d-spectrum. Then
\[
\tau_+^k=\tau_- \quad\text{and}\quad \tau_-^k=\tau_+.
\]
Thus $(\operatorname{dSpec}\mathcal{L},\tau_+,\tau_-)$ is its own patch space.
\end{prop}

\begin{proof} 
Since the family $\{\phi_+(a):a\in L_+\}$ of compact open subsets of $(\operatorname{dSpec}\mathcal{L},\tau_+)$ forms a basis for $\tau_+$, it is easily verified that $\tau_+^k$ is generated by the sets
\[
\operatorname{dSpec}\mathcal{L}\setminus \phi_+(a),\qquad a\in L_+.
\]
Since $\mathcal{L}$ is a d-Boolean algebra, Lemma \ref{prime d-filter} gives the identity
\[
\phi_-(a^\dag)=\operatorname{dSpec}\mathcal{L}\setminus \phi_+(a)
\]
for each $a\in L_+$. Thus the collection
\[
\{\phi_-(a^\dag): a\in L_+\}
\]
forms a basis for $\tau_+^k$. But because $\mathcal{L}$ is a d-Boolean algebra, this same collection is a basis for $\tau_-$. Therefore $\tau_+^k=\tau_-$. The equality $\tau_-^k=\tau_+$ follows analogously.
\end{proof}

To proceed, we need the notion of subalgebra for d-Boolean algebras.

Suppose that $\mathcal{L}=(L;{t\!t},{f\!\!f};{\bf con}_L,{\bf tot}_L)$ is a d-Boolean algebra and that $\mathcal{N} =(N;{t\!t},{f\!\!f};{\bf con}_N,{\bf tot}_N)$ is a sub-d-lattice of $\mathcal{L}$. If the d-complement of every $x\in N_+$ belongs to $N_-$ and the d-complement of every $y\in N_-$ belongs to $N_+$, then $\mathcal{N}$ is also a d-Boolean algebra, and is called a \emph{subalgebra} of the d-Boolean algebra $\mathcal{L}$. It is clear that the intersection of any family of subalgebras of $\mathcal{L}$ is also a subalgebra of $\mathcal{L}$.

For each d-lattice $\mathcal{L}=(L;{t\!t},{f\!\!f};{\bf con},{\bf tot})$, we write
\[
(\operatorname{dSpec}\mathcal{L},\delta_+,\delta_-)
\]
for the patch space of the d-spectrum $(\operatorname{dSpec}\mathcal{L},\tau_+,\tau_-)$ of $\mathcal{L}$. That is, $\delta_+$ is generated as a subbasis by
\[
\{\phi_+(a):a\in L_+\}\cup \{\operatorname{dSpec}\mathcal{L}\setminus \phi_-(b): b\in L_-\},
\]
and $\delta_-$ by
\[
\{\phi_-(b):b\in L_-\}\cup \{\operatorname{dSpec}\mathcal{L}\setminus \phi_+(a): a\in L_+\}.
\]

For all $a\in L_+$ and $b\in L_-$, the set $\phi_+(a)$ is $\delta_+$-open and $\delta_-$-closed, while $\phi_-(b)$ is $\delta_-$-open and $\delta_+$-closed. Hence the correspondence
\[
(a,b)\mapsto (\phi_+(a),\phi_-(b))
\]
defines a d-lattice homomorphism from $\mathcal{L}$ to the d-Boolean algebra of d-clopen sets of $(\operatorname{dSpec}\mathcal{L},\delta_+,\delta_-)$.

Let $\Gamma\mathcal{L}$ be the smallest subalgebra of the d-Boolean algebra of d-clopen sets of $(\operatorname{dSpec}\mathcal{L},\delta_+,\delta_-)$ that contains
\[
\{(\phi_+(a),\emptyset):a\in L_+\}\cup \{(\emptyset,\phi_-(b)):b\in L_-\}.
\]
Write the underlying lattice of $\Gamma\mathcal{L}$ as $D_+\times D_-$. It is not hard to check that every element of $D_+$ is a finite join of sets of the form
\[
\phi_+(a)\cap(\operatorname{dSpec}\mathcal{L}\setminus \phi_-(b)),\quad a\in L_+,\ b\in L_-,
\]
and similarly every element of $D_-$ is a finite join of sets of the form
\[
\phi_-(b)\cap(\operatorname{dSpec}\mathcal{L}\setminus \phi_+(a)),\quad a\in L_+,\ b\in L_-.
\]

\begin{prop}\label{reflective}
For each d-lattice $\mathcal{L}=(L;{t\!t},{f\!\!f};{\bf con},{\bf tot})$, the map
\[
\gamma_\mathcal{L}\colon \mathcal{L}\to \Gamma\mathcal{L},\quad
(a,b)\mapsto (\phi_+(a),\phi_-(b))
\]
is a d-lattice homomorphism, and the pair $(\Gamma\mathcal{L},\gamma_\mathcal{L})$ is a ${\bf dBA}$-reflection of $\mathcal{L}$.
\end{prop}

\begin{proof}
The map $(a,b)\mapsto(\phi_+(a),\phi_-(b))$ is clearly a lattice homomorphism from $L$ to the underlying lattice $D_+\times D_-$ of $\Gamma\mathcal{L}$. To verify that it preserves ${\bf con}$ and ${\bf tot}$, suppose $a\in L_+$ and $b\in L_-$ with $a\sqcup b\in{\bf con}$. For any prime d-filter $g$, we have $g(a\sqcup b)\in\{0,{t\!t},{f\!\!f}\}$, so either $g\notin\phi_+(a)$ or $g\notin\phi_-(b)$; hence $\phi_+(a)\cap\phi_-(b)=\emptyset$. Thus $\gamma_\mathcal{L}$ preserves ${\bf con}$. Preservation of ${\bf tot}$ is analogous.

Now we show the universal property. Let $\mathcal{M}$ be a d-Boolean algebra and let $f\colon\mathcal{L}\to\mathcal{M}$ be a d-lattice homomorphism. We need a unique d-Boolean algebra homomorphism $\overline{f}\colon \Gamma\mathcal{L}\to\mathcal{M}$ such that $\overline{f}\circ\gamma_\mathcal{L}=f$.

Existence: The map $\operatorname{dSpec} f$, sending each prime d-filter $g$ of $\mathcal{M}$ to $g\circ f$, is a d-spectral map from $\operatorname{dSpec}\mathcal{M}$ to $\operatorname{dSpec}\mathcal{L}$, and induces a continuous map between the respective patch spaces. Since $\mathcal{M}$ is a d-Boolean algebra, $\operatorname{dSpec}\mathcal{M}$ is its own patch space, and the d-Boolean algebra of its d-clopen sets is $\mathcal{M}$. Therefore $(\operatorname{dSpec} f)^{-1}$ gives a d-Boolean algebra homomorphism from the d-clopen sets of $(\operatorname{dSpec}\mathcal{L},\delta_+,\delta_-)$ to $\mathcal{M}$. Restricting to $\Gamma\mathcal{L}$ yields the desired $\overline{f}$.

Uniqueness: If $h_1,h_2\colon\Gamma\mathcal{L}\to\mathcal{M}$ satisfy $h_1\circ\gamma_\mathcal{L}=f=h_2\circ\gamma_\mathcal{L}$, then the subset of $D_+\times D_-$ on which they agree is a sub-d-Boolean algebra containing all generators, so $h_1=h_2$.
\end{proof}

Therefore, the category of d-Boolean algebras is a simultaneously reflective and coreflective full subcategory of the category of d-lattices. Together with the fact that Boolean algebras form a simultaneously reflective and coreflective full subcategory of the category ${\bf DisLat}$ of distributive lattices, this exhibits the category of d-lattices as a well-behaved extension of that of Boolean algebras:
\[
{\bf BA}\subseteq {\bf DisLat}\cong {\bf dBA}\subseteq {\bf dLat}.
\]

\begin{thm}\label{spectrum of Gamma = patch}
For every d-lattice, the d-spectrum of its d-Boolean algebra reflection coincides with the patch bitopological space of its d-spectrum.
\end{thm}

\begin{proof}
Let $\mathcal{L}=(L;{t\!t},{f\!\!f};{\bf con},{\bf tot})$ be a d-lattice and let $(\operatorname{dSpec}\mathcal{L},\tau_+,\tau_-)$ be its d-spectrum. Let $(\Gamma\mathcal{L},\gamma_\mathcal{L})$ be the ${\bf dBA}$-reflection, and let $(\operatorname{dSpec}\Gamma\mathcal{L},\nu_+,\nu_-)$ be the d-spectrum of the d-Boolean algebra $\Gamma\mathcal{L}$.

The universal property of $\gamma_\mathcal{L}$ implies that every prime d-filter $g$ of $\mathcal{L}$ extends uniquely to a prime d-filter $\overline{g}$ of $\Gamma\mathcal{L}$ such that $g=\overline{g}\circ\gamma_\mathcal{L}$. Hence the map
\[
\operatorname{dSpec}\gamma_\mathcal{L}\colon \operatorname{dSpec}\Gamma\mathcal{L}\to \operatorname{dSpec}\mathcal{L},\quad
\overline{g}\mapsto \overline{g}\circ\gamma_\mathcal{L}
\]
is a continuous bijection. By Proposition \ref{spectrum of d-boolean algebra is d-Boolean}, the d-spectrum of $\Gamma\mathcal{L}$ is its own patch space, so $\operatorname{dSpec}\gamma_\mathcal{L}$ is a continuous bijection from $(\operatorname{dSpec}\Gamma\mathcal{L},\nu_+,\nu_-)$ to the patch space $(\operatorname{dSpec}\mathcal{L},\delta_+,\delta_-)$. We show it is a homeomorphism of bitopological spaces by proving that both
\[
\operatorname{dSpec}\gamma_\mathcal{L}\colon (\operatorname{dSpec}\Gamma\mathcal{L},\nu_+)\to(\operatorname{dSpec}\mathcal{L},\delta_+)
\]
and the analogous map for $\nu_-$ are open.

Since every element of $D_+$ is a finite join of sets of the form
\[
U=\phi_+(a)\cap(\operatorname{dSpec}\mathcal{L}\setminus\phi_-(b)),\quad a\in L_+,\ b\in L_-,
\]
it suffices to show that the image under $\operatorname{dSpec}\gamma_\mathcal{L}$ of
\[
\phi_+(U)=\{\overline{g}\in\operatorname{dSpec}\Gamma\mathcal{L}: \overline{g}(U,\emptyset)={t\!t}\}
\]
belongs to $\delta_+$.
In the d-Boolean algebra $\Gamma\mathcal{L}$, the element $(\emptyset,\phi_-(b))$ is the d-complement of $(\operatorname{dSpec}\mathcal{L}\setminus\phi_-(b),\emptyset)$. By Lemma \ref{prime d-filter}, for every prime d-filter $\overline{g}$ of $\Gamma\mathcal{L}$,
\[
\overline{g}(\emptyset,\phi_-(b))=0 \iff \overline{g}(\operatorname{dSpec}\mathcal{L}\setminus\phi_-(b),\emptyset)={t\!t}.
\]
Thus, for $\overline{g}\in\operatorname{dSpec}\Gamma\mathcal{L}$ and $g=\overline{g}\circ\gamma_\mathcal{L}$,
\begin{align*}
\overline{g}\in\phi_+(U)
&\iff \overline{g}(U,\emptyset)={t\!t}\\
&\iff \overline{g}(\phi_+(a),\emptyset)={t\!t}
~~\text{and}~~
\overline{g}(\operatorname{dSpec}\mathcal{L}\setminus\phi_-(b),\emptyset)={t\!t}\\
&\iff \overline{g}(\phi_+(a),\emptyset)={t\!t}
~~\text{and}~~
\overline{g}(\emptyset,\phi_-(b))=0\\
&\iff g\in\phi_+(a)~~ \text{and}~~  g\in\operatorname{dSpec}\mathcal{L}\setminus\phi_-(b)\\
&\iff \operatorname{dSpec}\gamma_\mathcal{L}(\overline{g})\in U.
\end{align*}
Hence the image of $\phi_+(U)$ is $U$, which belongs to $\delta_+$. Therefore $\operatorname{dSpec}\gamma_\mathcal{L}$ is open for $\nu_+$; the argument for $\nu_-$ is similar.
\end{proof}

\begin{prop}\label{spectrum is compact}
The d-spectrum $(\operatorname{dSpec}\mathcal{L},\tau_+,\tau_-)$ of every d-lattice $\mathcal{L}$ is compact.
\end{prop}
\begin{proof}
By Theorem \ref{spectrum of Gamma = patch}, the patch space of $(\operatorname{dSpec}\mathcal{L},\tau_+,\tau_-)$ is d-Boolean, hence compact. Since $(\operatorname{dSpec}\mathcal{L},\tau_+,\tau_-)$ is coarser than its patch space, it is necessarily a compact bitopological space itself.
\end{proof}

\section{A duality theorem for d-spectral spaces}\label{duality}

In this section we prove that the category of d-spectral spaces and d-spectral maps is dually equivalent to a full subcategory of d-lattices, namely the subcategory consisting of those d-lattices that satisfy the finitary cut rules.

For each d-lattice homomorphism $f\colon\mathcal{L}\to\mathcal{M}$, define
\[
\operatorname{dSpec} f\colon \operatorname{dSpec}\mathcal{M}\to \operatorname{dSpec}\mathcal{L}
\]
by $g\mapsto g\circ f$. It is readily verified that for all $a\in L_+$ and $b\in L_-$,
\[
(\operatorname{dSpec} f)^{-1}(\phi_+(a))=\phi_+(f(a)),\quad
(\operatorname{dSpec} f)^{-1}(\phi_-(b))=\phi_-(f(b)).
\]
By Proposition \ref{varphi is injective}, $\operatorname{dSpec} f\colon \operatorname{dSpec}\mathcal{M}\to\operatorname{dSpec}\mathcal{L}$ is a d-spectral map. This gives a contravariant functor
\[
\operatorname{dSpec}\colon {\bf dLat}^{\rm op}\to  {\bf dSpectral}.
\]

\begin{prop}\label{dSpec & KO}
The functor
\[
\operatorname{dSpec}\colon {\bf dLat}^{\rm op}\to{\bf dSpectral}
\]
is right adjoint to the functor
\[
\mathrm{K}\mathcal{O}\colon {\bf dSpectral}\to{\bf dLat}^{\rm op}.
\]
\end{prop}
\begin{proof}
We show that for every d-lattice $\mathcal{L}$ and every d-spectral space $X$, there is a natural bijection between d-spectral maps $X\to\operatorname{dSpec}\mathcal{L}$ and d-lattice homomorphisms $\mathcal{L}\to\mathrm{K}\mathcal{O}(X)$.

Define
\[
\upsilon_\mathcal{L}\colon
\mathcal{L}\to\mathrm{K}\mathcal{O}\circ\operatorname{dSpec}\mathcal{L}
\]
by the assignment
\[
a\sqcup b\mapsto\bigl(\phi_+(a),\phi_-(b)\bigr)
\]
for all $a\in L_+$ and $b\in L_-$.

Since $\phi_+$ and $\phi_-$ preserve finite lattice operations and bounds, this assignment preserves the lattice structure and the distinguished elements ${t\!t}$ and ${f\!\!f}$. Moreover, the definition of prime d-filters implies that
\[
a\sqcup b\in{\bf con}
  \implies 
\phi_+(a)\cap\phi_-(b)=\emptyset,
\qquad
a\sqcup b\in{\bf tot}
  \implies 
\phi_+(a)\cup\phi_-(b)=\operatorname{dSpec}\mathcal L.
\]
Thus
\[
\upsilon_\mathcal{L}\colon
\mathcal{L}\to\mathrm{K}\mathcal{O}\circ\operatorname{dSpec}\mathcal{L}
\]
is a d-lattice homomorphism.

The correspondence
\[
f\colon X\to\operatorname{dSpec}\mathcal{L}
\; \mapsto \;
\mathrm{K}\mathcal{O}(f)\circ\upsilon_\mathcal{L}\colon \mathcal{L}\to\mathrm{K}\mathcal{O}(X)
\]
is then the desired bijection.
\end{proof}

The adjunction $\mathrm{K}\mathcal{O}\dashv \operatorname{dSpec}$ will be used to establish the duality for d-spectral spaces, so we spell out its unit and counit:
\begin{itemize}
\item For each d-lattice $\mathcal{L}$, the counit map is the d-lattice homomorphism
\[
\upsilon_\mathcal{L}\colon \mathcal{L}\to\mathrm{K}\mathcal{O}\circ\operatorname{dSpec}\mathcal{L}
\]
given in the proof of Proposition \ref{dSpec & KO}.
\item For each d-spectral space $(X,\tau_+,\tau_-)$, the unit map is given by
\[
\sigma_X\colon (X,\tau_+,\tau_-)\to \operatorname{dSpec}\circ \mathrm{K}\mathcal{O} (X,\tau_+,\tau_-),
\qquad
x\mapsto [x]\colon \mathrm{K}\mathcal{O} (X,\tau_+,\tau_-)\to \mathbb{B},
\]
where for all $(U,V)\in \mathrm{K}\mathcal{O} (X,\tau_+,\tau_-)$,
\[
[x](U, V)=
\begin{cases}
1, & x\in U\cap V,\\
{t\!t}, & x\in U\setminus V,\\
{f\!\!f}, & x\in V\setminus U,\\
0, & x\notin U\cup V.
\end{cases}
\]
\end{itemize}

\begin{prop}\label{spectrum is spectral}
For each d-spectral space $(X,\tau_+,\tau_-)$, the map
\[
\sigma_X\colon (X,\tau_+,\tau_-)\to \operatorname{dSpec}\circ \mathrm{K}\mathcal{O} (X,\tau_+,\tau_-)
\]
is a homeomorphism. Therefore, d-spectral spaces are precisely the d-spectra of d-lattices.
\end{prop}
\begin{proof}
By Proposition \ref{spectrum as composite}, we have
\[
\operatorname{dSpec}\circ\mathrm{K}\mathcal{O}(X,\tau_+,\tau_-)
\cong
\operatorname{dpt} \circ\operatorname{Idl}\circ\mathrm{K}\mathcal{O}(X,\tau_+,\tau_-).
\]
Since $(X,\tau_+,\tau_-)$ is d-sober, in order to see that
\[
\sigma_X\colon (X,\tau_+,\tau_-)\to \operatorname{dpt} \circ\operatorname{Idl}\circ \mathrm{K}\mathcal{O} (X,\tau_+,\tau_-)
\]
is a homeomorphism, it suffices to show that its transpose (with respect to the adjunction $\mathrm{d}\mathcal{O}\dashv \operatorname{dpt}$) is an isomorphism of d-frames; that is, the d-frame homomorphism
\[
\kappa\colon \operatorname{Idl}\circ \mathrm{K}\mathcal{O} (X,\tau_+,\tau_-) \to  \mathrm{d}\mathcal{O}(X,\tau_+,\tau_-),
\qquad
(I_+,I_-)\mapsto \Big(\bigcup_{U\in I_+}U,\;\bigcup_{V\in I_-}V\Big)
\]
is an isomorphism.

Since both $\tau_+$ and $\tau_-$ are coherent frames, $\kappa$ is a frame isomorphism. It remains to check that it preserves and reflects both predicates. The argument is essentially the same as that in \cite[Proposition 4.9]{YZ2026}, we include it here for convenience of the reader.

For consistency, if $(I_+,I_-)$ belongs to the consistency predicate, then $U\cap V=\emptyset$ for all $U\in I_+$ and $V\in I_-$, hence
\[
\Big(\bigcup_{U\in I_+}U\Big)\cap\Big(\bigcup_{V\in I_-}V\Big)=\emptyset,
\]
showing that $\kappa(I_+,I_-)\in{\bf con}$.
Conversely, if $\kappa(I_+,I_-)\in{\bf con}$, then for any compact open $U\subseteq \bigcup I_+$ and compact open $V\subseteq \bigcup I_-$, we have $U\cap V=\emptyset$, so the pair $(I_+,I_-)$ belongs to the consistency predicate of $\operatorname{Idl}\circ\mathrm{K}\mathcal{O}$.

For totality, if $(I_+,I_-)$ belongs to the totality predicate, then there exist
$U\in I_+$ and $V\in I_-$ such that $U\cup V=X$, hence
$\kappa(I_+,I_-)\in\mathbf{tot}$.
Conversely, if $\kappa(I_+,I_-)\in\mathbf{tot}$, then
\[
\Big(\bigcup_{U\in I_+}U\Big)
\cup
\Big(\bigcup_{V\in I_-}V\Big)
=X.
\]
By compactness of $(X,\tau_+\vee\tau_-)$, finitely many members of
$I_+\cup I_-$ cover $X$. Since ideals are closed under finite joins, there exist $U_0\in I_+$ and
$V_0\in I_-$ such that $U_0\cup V_0=X$. Thus $(I_+,I_-)$ belongs to the
totality predicate of  $\operatorname{Idl}\circ\mathrm{K}\mathcal{O}(X,\tau_+,\tau_-)$.
\end{proof}

To determine when the counit map
\[
\upsilon_\mathcal{L}\colon \mathcal{L}\to\mathrm{K}\mathcal{O}\circ\operatorname{dSpec}\mathcal{L}
\]
is an isomorphism of d-lattices, we need the notion of finitary cut rules for d-lattices.

\begin{defn}[Finitary cut rules] (\cite[Section 5.2]{JM2006}, \cite[Section 2.5]{Klinke})
Let $\mathcal{L}=(L;{t\!t},{f\!\!f};{\bf con},{\bf tot})$ be a d-lattice.
We say that $\mathcal{L}$ satisfies the finitary cut rules if for all $x,a\in L_+$ and $y,b\in L_-$:
\begin{enumerate}[leftmargin=5em]
\item[$({\rm cut}_{\bf con})$]
$(x\sqcap a)\sqcup y\in {\bf con},\; x\sqcup (y\sqcap b)\in {\bf con},\;
a\sqcup b\in{\bf tot} \implies x\sqcup y\in{\bf con}.$
\item[$({\rm cut}_{\bf tot})$]
$(x\sqcup a)\sqcup y\in {\bf tot},\; x\sqcup (y\sqcup b)\in {\bf tot},\;
a\sqcup b\in{\bf con} \implies x\sqcup y\in{\bf tot}.$
\end{enumerate}
\end{defn}

It is clear that if a d-lattice satisfies the finitary cut rules, then so do all of its sub-d-lattices.

\begin{exmp} \label{spatial implies distributive}
For every set $X$, the d-frame of all subsets of $X$ (Example \ref{d-frame of subsets}) satisfies the finitary cut rules by direct verification; hence so does every spatial d-frame, as already observed in the argument of \cite[Theorem 3.2.6]{Klinke}.
\end{exmp}

The finitary cut rules are self-dual in the sense that a d-lattice $\mathcal{L}=(L;{t\!t},{f\!\!f};{\bf con},{\bf tot})$ satisfies the finitary cut rules if and only if so does its order-dual $\mathcal{L}^\partial$. Recall that $L^\partial$ carries the opposite order of $L$ and
\[
L^\partial_+=\{a\sqcup{f\!\!f}:a\in L_+\}, \qquad
L^\partial_-=\{{t\!t}\sqcup b:b\in L_-\}.
\]
Write
\[
a^\partial=a\sqcup{f\!\!f}\in L^\partial_+,
\qquad
b^\partial={t\!t}\sqcup b\in L^\partial_-.
\]
Then for all $x,a\in L_+$ and $y,b\in L_-$, by definition we have
\[
x^\partial\sqcup^\partial a^\partial = (x\sqcap a)^\partial,
\qquad
y^\partial\sqcap^\partial b^\partial = (y\sqcup b)^\partial
\]
and
\[
x^\partial \sqcup^\partial y^\partial
= (x\sqcup{f\!\!f}) \sqcap ({t\!t} \sqcup y)
=  x\sqcup y.
\]

\begin{lem} \label{FCR is self dual}
Let $\mathcal{L}=(L;{t\!t},{f\!\!f};{\bf con},{\bf tot})$ be a d-lattice. Then $\mathcal{L}$ satisfies the finitary cut rules if and only if so does its order-dual $\mathcal{L}^\partial$.
\end{lem}
\begin{proof}
For all $x,a\in L_+$ and $y,b\in L_-$, it holds that
\[
\begin{cases}
(x^\partial\sqcup^\partial a^\partial)\sqcup^\partial y^\partial
   =(x\sqcap a)\sqcup y,\\
x^\partial\sqcup^\partial(y^\partial\sqcup^\partial b^\partial)
   =x\sqcup(y\sqcap b), \\
a^\partial \sqcup^\partial b^\partial   =  a\sqcup b, \\
   x^\partial \sqcup^\partial y^\partial  =  x\sqcup y.
\end{cases}
\]
Thus, the premises and conclusion of ${\rm cut}_{\bf tot}$ in $\mathcal L^\partial$ correspond exactly to those of ${\rm cut}_{\bf con}$ in $\mathcal L$, showing that $\mathcal L$ satisfies ${\rm cut}_{\bf con}$ if and only if $\mathcal L^\partial$ satisfies ${\rm cut}_{\bf tot}$. Likewise, $\mathcal L^\partial$ satisfies ${\rm cut}_{\bf con}$ if and only if $\mathcal L$ satisfies ${\rm cut}_{\bf tot}$. Therefore, $\mathcal{L}=(L;{t\!t},{f\!\!f};{\bf con},{\bf tot})$ satisfies the finitary cut rules if and only if so does $\mathcal{L}^\partial$.
\end{proof}

The following lemma shows that the finitary cut rules guarantee the existence of enough prime d-filters.

\begin{lem}\label{existence of prime d-filter}
Let $\mathcal{L}=(L;{t\!t},{f\!\!f};{\bf con},{\bf tot})$ be a d-lattice that satisfies the finitary cut rules; let $x\in L_+$ and let $y\in L_-$.
\begin{enumerate}[label=(\roman*)]
\item If $x\sqcup y\notin {\bf con}$, then there is a prime d-filter $g\colon\mathcal{L}\to\mathbb{B}$ such that $g(x\sqcup y)=1$.
\item If $x\sqcup y\notin{\bf tot}$, then there is a prime d-filter $g\colon\mathcal{L}\to\mathbb{B}$ such that $g(x\sqcup y)=0$.
\end{enumerate}
\end{lem}

\begin{proof}
(i) Let $\mathcal{C}$ be the set of all pairs $(F_+,F_-)$ of proper filters of $L_+$ and $L_-$, respectively, satisfying:
\begin{enumerate}[label=(\alph*)]
\item $x\in F_+$ and $y\in F_-$;
\item $u\sqcup v\notin {\bf con}$ for every $u\in F_+$ and $v\in F_-$.
\end{enumerate}

The pair $(\uparrow\! x,\uparrow\! y)$ belongs to $\mathcal{C}$. Indeed, both filters are proper, since $x\sqcup y\notin{\bf con}$ implies $x\ne 0$ and $y\ne 0$. If $u\in\;\uparrow\! x$ and $v\in\;\uparrow\! y$ satisfied $u\sqcup v\in{\bf con}$, then $x\sqcup y\in{\bf con}$ would follow from the downward closure of ${\bf con}$, a contradiction. Thus $\mathcal{C}\ne\varnothing$.

Order $\mathcal{C}$ by componentwise inclusion. The componentwise union of any chain in $\mathcal{C}$ is again a pair in $\mathcal{C}$, hence every chain has an upper bound in $\mathcal{C}$. By Zorn's lemma, $\mathcal{C}$ has a maximal element, say $(P_+,P_-)$.

If $u\notin P_+$, maximality implies that the filter generated by $P_+\cup\{u\}$ is either improper or contains an element $r$ such that $r\sqcup q\in{\bf con}$ for some $q\in P_-$. In either case, there exist $p\in P_+$ and $q\in P_-$ such that
\[
(p\sqcap u)\sqcup q\in{\bf con}.
\]
Dually, if $v\notin P_-$, there exist $p\in P_+$ and $q\in P_-$ such that
\[
p\sqcup(q\sqcap v)\in{\bf con}.
\]

Now we proceed in three steps.

{\bf Step 1.}   $P_+$ is a prime filter of $L_+$. Suppose that $r\sqcup s\in P_+$ while $r,s\notin P_+$. The preceding witnesses give $p,p'\in P_+$ and $q,q'\in P_-$ such that
\[
(p\sqcap r)\sqcup q\in{\bf con},\qquad
(p'\sqcap s)\sqcup q'\in{\bf con}.
\]
Put $p_0=p\sqcap p'$ and $q_0=q\sqcap q'$. Since ${\bf con}$ is a lower set under the information order $\sqsubseteq$,
\[
(p_0\sqcap r)\sqcup q_0\in{\bf con},\qquad
(p_0\sqcap s)\sqcup q_0\in{\bf con}.
\]
Since ${\bf con}$ is a sublattice of $L$ under the logic order, i.e., a sublattice of $(L,\leq,\wedge,\vee)$,
\[
\bigl((p_0\sqcap r)\sqcup q_0\bigr)\vee \bigl((p_0\sqcap s)\sqcup q_0\bigr)
=\bigl((p_0\sqcap r)\sqcup(p_0\sqcap s)\bigr)\sqcup q_0
=\bigl(p_0\sqcap(r\sqcup s)\bigr)\sqcup q_0\in{\bf con}.
\]
But $p_0\sqcap(r\sqcup s)\in P_+$ and $q_0\in P_-$, contradicting the condition (b). Hence $P_+$ is prime.

Likewise, one checks that $P_-$ is a prime filter of $L_-$.

{\bf Step 2.} For all $u\in L_+$ and $v\in L_-$, if $u\sqcup v\in{\bf tot}$, then $u\in P_+$ or $v\in P_-$. Otherwise, maximality gives $p,p'\in P_+$ and $q,q'\in P_-$ such that
\[
(p\sqcap u)\sqcup q\in{\bf con},\qquad
p'\sqcup(q'\sqcap v)\in{\bf con}.
\]
Put $p_0=p\sqcap p'$ and $q_0=q\sqcap q'$. Downward closure gives
\[
(p_0\sqcap u)\sqcup q_0\in{\bf con},\qquad
p_0\sqcup(q_0\sqcap v)\in{\bf con}.
\]
The rule $\mathrm{cut}_{\bf con}$ now implies $p_0\sqcup q_0\in{\bf con}$, contrary to condition~(b).

{\bf Step 3.} Define $g\colon L\to \mathbb{B}$ by, for $a\in L_+$ and $b\in L_-$,
\[
g(a\sqcup b)=
\begin{cases}
1, & a\in P_+,\ b\in P_-,\\
{t\!t}, & a\in P_+,\ b\notin P_-,\\
{f\!\!f}, & a\notin P_+,\ b\in P_-,\\
0, & a\notin P_+,\ b\notin P_-.
\end{cases}
\]
One verifies readily that $g$ is a prime d-filter with $g(x\sqcup y)=1$, as desired.

(ii) Apply (i) to the order-dual of $\mathcal{L}$, which also satisfies the finitary cut rules by Lemma \ref{FCR is self dual}.
\end{proof}

\begin{prop}\label{cut rules implies isomorphism}
Let $\mathcal{L}=(L;{t\!t},{f\!\!f};{\bf con},{\bf tot})$ be a d-lattice. Then the following are equivalent:
\begin{enumerate}[label=(\arabic*)]
\item $\mathcal{L}$ satisfies the finitary cut rules.
\item The counit map
\[
\upsilon_{\mathcal{L}}\colon \mathcal{L}\to
\mathrm{K}\mathcal{O}\circ \operatorname{dSpec}\mathcal{L}
\]
is an isomorphism of d-lattices.
\item The d-frame $\operatorname{Idl}\mathcal{L}$ is spatial.
\end{enumerate}
\end{prop}

\begin{proof}
$(1)\Rightarrow(2)$. Suppose that $\mathcal L$ satisfies the finitary cut rules. By Proposition \ref{varphi is injective}, the compact open sets of $(\operatorname{dSpec}\mathcal{L},\tau_+)$ are precisely the sets $\phi_+(a)$ with $a\in L_+$, and analogously for $(\operatorname{dSpec}\mathcal{L},\tau_-)$. Thus the counit
\[
\upsilon_{\mathcal L}\colon \mathcal{L}\to \mathrm{K}\mathcal{O}\circ \operatorname{dSpec}\mathcal{L}
\]
is an isomorphism on the underlying lattices. Since $\upsilon_\mathcal{L}$ is already a d-lattice homomorphism, it remains to show that it also reflects both ${\bf con}$ and ${\bf tot}$.

Let $x\in L_+$ and $y\in L_-$. If $x\sqcup y\notin{\bf con}$, Lemma \ref{existence of prime d-filter}\,(i) gives a prime d-filter $g$ with $g(x\sqcup y)=1$, which implies $g(x)={t\!t}$ and $g(y)={f\!\!f}$, hence $g\in\phi_+(x)\cap\phi_-(y)$. Therefore $\upsilon_{\mathcal L}(x\sqcup y)$, which is the pair $\bigl(\phi_+(x),\phi_-(y)\bigr)$, does not belong to the consistency predicate of $\mathrm{K}\mathcal{O}\circ \operatorname{dSpec}\mathcal{L}$, showing that $\upsilon_\mathcal{L}$ reflects ${\bf con}$. If $x\sqcup y\notin{\bf tot}$, Lemma \ref{existence of prime d-filter}\,(ii) gives a prime d-filter $g$ with $g(x\sqcup y)=0$, which implies that $g\notin\phi_+(x)\cup\phi_-(y)$. Therefore $\upsilon_{\mathcal L}(x\sqcup y)$ does not belong to the totality predicate of $\mathrm{K}\mathcal{O}\circ \operatorname{dSpec}\mathcal{L}$, showing that $\upsilon_\mathcal{L}$ reflects ${\bf tot}$.

$(2)\Rightarrow(3)$. Since
\[
\upsilon_{\mathcal L}\colon\mathcal L\to
\mathrm{K}\mathcal{O}\circ\operatorname{dSpec}\mathcal L
\]
is an isomorphism, applying the functor $\operatorname{Idl}$ yields
\[
\operatorname{Idl}\mathcal L
\cong
\operatorname{Idl}\circ\mathrm{K}\mathcal{O}
\circ\operatorname{dSpec}\mathcal L.
\]
Recall that the proof of Proposition \ref{spectrum is spectral} establishes, for every d-spectral space $X$, a d-frame isomorphism
\[
\operatorname{Idl}\circ\mathrm{K}\mathcal{O}(X)
\cong
\mathrm{d}\mathcal{O}(X).
\]
Since $\operatorname{dSpec}\mathcal L$ is d-spectral, applying this isomorphism to $X=\operatorname{dSpec}\mathcal L$ yields
\[
\operatorname{Idl}\circ\mathrm{K}\mathcal{O}
\circ\operatorname{dSpec}\mathcal L
\cong \mathrm{d}\mathcal{O} \circ \operatorname{dSpec}\mathcal L.
\]
Thus $\operatorname{Idl}\mathcal L$ is isomorphic to the d-frame of open sets of $\operatorname{dSpec}\mathcal L$ and is therefore spatial.

$(3)\Rightarrow(1)$. The argument of \cite[Proposition 4.9]{YZ2026} establishes that $\mathcal L$ is isomorphic to a sub-d-lattice of $\operatorname{Idl}\mathcal L$. If $\operatorname{Idl}\mathcal L$ is a spatial d-frame, it satisfies the finitary cut rules by Example \ref{spatial implies distributive}, hence so does $\mathcal L$.
\end{proof}

\begin{cor}
\begin{enumerate}[label=(\roman*)]
\item Every d-Boolean algebra satisfies the finitary cut rules.
\item {\rm(\cite[Theorem 3.2.6]{Klinke})} A finite d-lattice satisfies the finitary cut rules if and only if, as a d-frame, it is spatial.
\end{enumerate}
\end{cor}

\begin{proof}
(i) This follows from Proposition \ref{cut rules implies isomorphism} and the fact that $\operatorname{Idl}\mathcal{L}$ is spatial for every d-Boolean algebra $\mathcal{L}$ (Proposition \ref{spectra of d-Boolean algebras}).

(ii) By the argument of \cite[Proposition 4.9]{YZ2026}, every d-lattice $\mathcal{L}$ is a sub-d-lattice of the d-frame $\operatorname{Idl}\mathcal L$, hence a sub-d-frame of $\operatorname{Idl}\mathcal L$ when $\mathcal{L}$ is finite. The conclusion then follows from Proposition \ref{cut rules implies isomorphism}.
\end{proof}

\begin{thm}\footnote{This theorem was added in revision: we discovered it after the first version had been submitted to the journal. Yifan Meng (Sichuan University) independently obtained the same result after reading the first version.} \label{the duality}
The category ${\bf dSpectral}$ of d-spectral spaces and d-spectral maps and the category ${\bf dLat}_{\rm cut}$ of d-lattices satisfying the finitary cut rules and d-lattice homomorphisms are dually equivalent.
\end{thm}

\begin{proof}
Consider the adjunction
\[
\begin{tikzcd}[column sep=4.5em]
{\bf dSpectral}
  \arrow[r,shift left=1.35ex,"\mathrm{K}\mathcal{O}"]
   \arrow[r,phantom,"\bot" description]
&
{\bf dLat}^{\rm op}.
 \arrow[l,shift left=1.35ex,"\operatorname{dSpec}"]
\end{tikzcd}
\]

On the one hand, for each d-spectral space $X$, the d-lattice $\mathrm{K}\mathcal{O}(X)$ satisfies the finitary cut rules since it is a sub-d-lattice of the d-frame of subsets of $X$. Moreover, the unit map
\[
\sigma_X\colon X\to \operatorname{dSpec}\circ \mathrm{K}\mathcal{O} (X)
\]
is a homeomorphism by Proposition \ref{spectrum is spectral}.

On the other hand, for each d-lattice $\mathcal{L}$, Proposition \ref{cut rules implies isomorphism} shows that the counit map
\[
\upsilon_{\mathcal{L}}\colon \mathcal{L}\to
\mathrm{K}\mathcal{O}\circ \operatorname{dSpec}\mathcal{L}
\]
is an isomorphism if and only if $\mathcal{L}$ satisfies the finitary cut rules.

Therefore, replacing ${\bf dLat}$ by its full subcategory ${\bf dLat}_{\rm cut}$ composed of d-lattices satisfying the finitary cut rules gives an equivalence of categories
\[
\begin{tikzcd}[column sep=4em]
({\bf dLat}_{\rm cut})^{\rm op}
  \arrow[r,shift left=0.75ex,"\operatorname{dSpec}"]
&
{\bf dSpectral}.
  \arrow[l,shift left=0.75ex,"\mathrm{K}\mathcal{O}"]
\end{tikzcd}
\]
This finishes the proof.
\end{proof}

\begin{cor}
For each d-lattice $\mathcal{L}$, the following are equivalent:
\begin{enumerate}[label=(\arabic*)]
\item $\mathcal{L}$ satisfies the finitary cut rules.
\item There exists some set $X$ such that $\mathcal{L}$ is isomorphic to a sub-d-lattice of the d-frame of subsets of $X$.
\item There exists some bitopological space $X$ such that $\mathcal{L}$ is isomorphic to a sub-d-lattice of the d-frame of open sets of $X$.
\end{enumerate}
\end{cor}

\begin{cor}\label{d-Boolean r.c.in d-spectral}
The category of d-Boolean spaces is a simultaneously reflective and coreflective full subcategory of the category of d-spectral spaces and d-spectral maps.
\end{cor}
\begin{proof}
Since all d-Boolean algebras satisfy the finitary cut rules and constitute a simultaneously reflective and coreflective full subcategory of the category of d-lattices, it follows that the category ${\bf dBA}$ is a simultaneously reflective and coreflective full subcategory of the category ${\bf dLat}_{\rm cut}$ composed of d-lattices satisfying the finitary cut rules.

Now consider the square
\[
\begin{tikzcd}[column sep=normal,row sep=huge]
{\bf dBA}^{\rm op}
\arrow[r] \arrow[d,"\operatorname{dSpec}"'] 
& ({\bf dLat}_{\rm cut})^{\rm op}
\arrow[d,"\operatorname{dSpec}"]
\\
 {\bf dBoolSp}
\arrow[r]
& {\bf dSpectral}
\end{tikzcd}
\]
where the horizontal arrows denote inclusion.

Since both vertical arrows are equivalences of categories, it follows that ${\bf dBoolSp}$ is a simultaneously reflective and coreflective full subcategory of ${\bf dSpectral}$.
\end{proof}

Explicitly, for each d-spectral space $(X,\tau_+,\tau_-)$:
\begin{itemize}
\item The d-Boolean space coreflection of $(X,\tau_+,\tau_-)$ is given by the d-spectrum of the ${\bf dBA}$-reflection of the d-lattice of compact open sets of $(X,\tau_+,\tau_-)$, which coincides with the patch space $(X,\tau_+\vee\tau_-^k,\; \tau_-\vee\tau_+^k)$ by Theorem \ref{spectrum of Gamma = patch}.
\item The d-Boolean space reflection of $(X,\tau_+,\tau_-)$ is given by the d-spectrum of the ${\bf dBA}$-coreflection of the d-lattice of compact open sets of $(X,\tau_+,\tau_-)$, or equivalently, the d-spectrum of the d-Boolean algebra of d-clopen sets of $(X,\tau_+,\tau_-)$.
\end{itemize}

Finally, we present a counterpart of Theorem \ref{patch space of spectral space} in the bitopological context:

\begin{thm}\label{patch is spectral}
For every d-spectral space $(X,\tau_+,\tau_-)$:
\begin{enumerate}[label=(\roman*)]
\item the de Groot dual $(X,\tau_+^k,\tau_-^k)$ is d-spectral;
\item the patch space $(X,\tau_+\vee\tau_-^k,\; \tau_-\vee\tau_+^k)$ is d-Boolean.
\end{enumerate}
\end{thm}
\begin{proof}
By Proposition \ref{spectrum is spectral}, $(X,\tau_+,\tau_-)$ is the d-spectrum of some d-lattice $\mathcal{L}$. By Lemma \ref{dspec preserves involution}, $(X,\tau_+^k,\tau_-^k)$ is the d-spectrum of the order-dual $\mathcal{L}^\partial$, hence d-spectral by Theorem \ref{functor dSpec}. This proves (i).

For (ii), by Theorem \ref{spectrum of Gamma = patch} the patch space
$(X,\tau_+\vee\tau_-^k,\; \tau_-\vee\tau_+^k)$
coincides with the d-spectrum of the d-Boolean algebra reflection of $\mathcal{L}$, hence is a d-Boolean space by Proposition \ref{spectra of d-Boolean algebras}.
\end{proof}

\section{Summary}\label{summary}
This paper introduces \emph{d-spectral spaces}, a bitopological analogue of spectral spaces, and develops a duality theorem for them. The main results are as follows: \begin{enumerate}[label=(\roman*)]
\item d-spectral spaces are precisely the d-spectra of d-lattices. 
\item  The category of d-spectral spaces is dually equivalent to the category of d-lattices satisfying the finitary cut rules. 
\item The category of spectral spaces embeds in the category of d-spectral spaces as a both reflective and coreflective full subcategory.
\item The de Groot dual of a d-spectral space is again d-spectral, while its patch space is d-Boolean. 
\end{enumerate}
Together, these results show that the category of d-spectral spaces is a well-behaved extension of the category of spectral spaces. The categorical relationships are summarized in the following
diagram:
\[
\begin{tikzcd}[column sep=normal,row sep=huge]
{\bf BA}
\arrow[r,"{\rm r.c.}"] \arrow[d,"{\rm d.e.}"'] 
& {\bf DisLat} 
\arrow[r,"\simeq"] \arrow[d,"{\rm d.e.}" description] 
& {\bf dBA}
\arrow[r,"{\rm r.c.}"] \arrow[d,"{\rm d.e.}" description] 
& {\bf dLat}_{\rm cut}
\arrow[d,"{\rm d.e.}"]
\\
{\bf BoolSp}
\arrow[r,"{\rm r.c.}"']
& {\bf Spectral}
\arrow[r,"\simeq"'] 
& {\bf dBoolSp}
\arrow[r,"{\rm r.c.}"']
& {\bf dSpectral}
\end{tikzcd}
\]
Here
\begin{itemize}
\item[$\bullet$] ${\rm r.c.}$ means simultaneously reflective and coreflective;
\item[$\bullet$] $\simeq$ means equivalent;
\item[$\bullet$] ${\rm d.e.}$ means dually equivalent.
\end{itemize}

\section*{Acknowledgements}

The authors are grateful to the referee for a careful reading and insightful suggestions, which led to substantial improvements to the paper. The first-named author acknowledges support from Jiangxi University of Science and Technology (Grant No. jxust-161). The second-named author acknowledges support from the National Natural Science Foundation of China (Grant No. 12371463).

\end{document}